\documentclass[10pt]{article}
\usepackage{mathrsfs}
\usepackage{threeparttable}
\usepackage{booktabs}
\usepackage{geometry}
\usepackage{epic}
\usepackage{epstopdf}
\usepackage{ytableau}
\usepackage[figuresright]{rotating}
\usepackage{mathtools}
\usepackage{float}
\usepackage{graphicx}
\usepackage{graphics}
\usepackage{amsfonts,color,amsmath,amssymb,amsthm,fancyhdr}
\usepackage{makecell}
\usepackage{multirow}
\usepackage{footnote}
\usepackage{appendix}

\newcommand{\tr}{\textup{Tr}}
\newcommand{\ra}{\rangle}
\newcommand{\la}{\langle}

\newcommand{\F}{{\mathbb F}}
\newcommand{\Z}{{\mathbb Z}}

\newcommand{\cP}{{\mathcal P}}
\newcommand{\cM}{{\mathcal M}}

\newcommand{\cQ}{{\mathcal Q}}

\newcommand{\bL}{{\mathbf L}}
\newcommand{\bO}{{\mathbf O}}
\newcommand{\bS}{{\mathbf S}}
\newcommand{\bU}{{\mathbf U}}
\newcommand{\f}{{\mathbf f}}

\newcommand{\GL}{\textup{GL}}

\newcommand{\SL}{\textup{SL}}
\newcommand{\SU}{\textup{SU}}
\newcommand{\GO}{\textup{GO}}

\newcommand{\PGL}{\textup{PGL}}

\newcommand{\PSL}{\textup{PSL}}
\newcommand{\PSU}{\textup{PSU}}
\newcommand{\PSp}{\textup{PSp}}

\newcommand{\PGU}{\textup{PGU}}
\newcommand{\GamL}{\Gamma\textup{L}}
\newcommand{\GamO}{\Gamma\textup{O}}
\newcommand{\GamU}{\Gamma\textup{U}}
\newcommand{\GamSp}{\Gamma\textup{Sp}}

\newcommand{\PGamSp}{\textup{P}\Gamma\textup{Sp}}
\newcommand{\PGamO}{\textup{P}\Gamma\textup{O}}
\newcommand{\POmeg}{\textup{P}\Omega}
\newcommand{\PGamU}{\textup{P}\Gamma\textup{U}}

\newcommand{\Gal}{\textup{Gal}}
\newcommand{\Aut}{\textup{Aut}}
\newcommand{\Sp}{\textup{Sp}}

\newtheorem{thm}{Theorem}

\newtheorem{lemma}[thm]{Lemma}

\newtheorem{proposition}[thm]{Proposition}

\newtheorem{example}[thm]{Example}
\numberwithin{equation}{section}
\numberwithin{thm}{section}
\numberwithin{table}{section}

\newtheorem{remark}[thm]{Remark}

\renewcommand{\paragraph}{\roman{paragraph}}

\usepackage{bm}
\usepackage{bbm}
\usepackage{hyperref}
\usepackage{tikz}
\usepackage{enumerate}
\usetikzlibrary{arrows,shapes,positioning}
\usetikzlibrary{decorations.markings}
\tikzstyle arrowstyle=[scale=1]
\tikzstyle directed=[postaction={decorate,decoration={markings, mark=at position .65 with {\arrow[arrowstyle]{stealth}}}}]
\tikzstyle reverse directed=[postaction={decorate,decoration={markings, mark=at position .65 with {\arrowreversed[arrowstyle]{stealth};}}}]

\usepackage{enumerate}
\usepackage{graphics}
\usepackage{graphicx}
\usepackage{subfigure}
\usepackage{epsfig}
\usepackage{bm}
\usepackage{lineno}
\usepackage{url}
\usepackage{float}
\usepackage{color}

\begin{document}
	
	\title{On $m$-ovoids of finite classical polar spaces with an irreducible transitive automorphism group}
	
	\author{Tao Feng$^{\text{a}}$, Weicong Li$^{\text{b}}$, Ran Tao$^{\text{c,d,}}$\thanks{Corresponding author}\\
		\footnotesize $^{\text{a}}$ School of Mathematical Sciences, Zhejiang University, Hangzhou 310027, China;\\
		\footnotesize $^{\text{b}}$ Department of Mathematics and and National Center for Applied Mathematics Shenzhen,\\
		\footnotesize Southern University of Science and Technology, ShenZhen 518055, China;\\
		\footnotesize $^{\text{c}}$ Key Laboratory of Cryptologic Technology and Information Security, Ministry of Education,\\
		\footnotesize Shandong University, Qingdao 266237, China;\\
		\footnotesize $^{\text{d}}$ School of Cyber Science and Technology, Shandong University, Qingdao 266237, China\\
		\footnotesize  Email: tfeng@zju.edu.cn, liwc3@sustech.edu.cn, rtao@sdu.edu.cn}
	
	\date{}
	\maketitle
	
	\begin{abstract}
		In this paper, we classify the  $m$-ovoids of finite classical polar spaces that admit a transitive automorphism group acting irreducibly on the ambient vector space.  In particular, we obtain  several new infinite families of transitive $m$-ovoids.
		
		
		\smallskip
		
		
	\end{abstract}
\section{Introduction}

The study of finite geometric structures that  exhibit certain regularity and satisfy certain transitivity conditions has been a main theme in finite geometry and has a long history. The celebrated Ostrom-Wagner theorem which asserts that a finite projective plane having a $2$-transitive collineation group must be classical dates back to 1959, cf. \cite{OstromProjective1959}. The classification of finite simple groups has greatly stimulated the study in this direction, and Aschbacher's Theorem \cite{AschbacherMaximal} on subgroup structures of finite classical groups provides a guideline for the study of symmetries of finite geometric structures.

In this paper, we aim to determine the $m$-ovoids of finite classical polar spaces that admit a transitive automorphism group acting irreducibly on the ambient space. An $m$-ovoid of a finite classical polar space $\cP$ is a point set of $\cP$ that meets each maximal totally isotropic or singular subspace in exactly $m$ points. A $1$-ovoid is simply an ovoid.
The concept of $m$-ovoids was first proposed by Shult and Thas in \cite{Shultmsystem}. It has been a very active research topic in recent years due to its close connection with graphs, codes and other geometric structures, cf. \cite{BambergTightsets2007,BambergTightquadrangles,cossidente2008m,Cossidente2018intri,FengOnm2020}. In \cite{BambergclassificationTransi}, Bamberg and Penttila classified transitive ovoids of finite classical polar spaces with an insoluble automorphism group. Recently, the first and second author completely classified the transitive ovoids of finite Hermitian polar spaces in \cite{FengLi}.

Our classification is based on the classification of subgroups of finite classical groups whose orders are divisible by certain primitive divisors in \cite{bamberg2008overgroups}, which in turn is based on \cite{Guralnick1999Linear}. One purpose of \cite{bamberg2008overgroups} is the application to the classification of geometric objects that satisfy certain transitivity condition, and this paper is a contribution in this direction.

This paper is organized as follows. In Section \ref{sec:Preli}, we present some preliminary results. In Section \ref{sec:MainResults}, we first list some examples of transitive $m$-ovoids and then present our main classification results, i.e., Theorems \ref{thm_PGamUTransi}-\ref{thm_PGamSpTransi}. In Section \ref{sec:boundsection}, we study $m$-ovoids of a particular form that admit a large automorphism group of extension field type by an exponential sum approach. We shall derive improved bounds on the parameters by Kloosterman sum estimations and obtain new infinite families. In Section  \ref{sec:mainresults} and Section \ref{sec_extField}, we prove our main results by going over the list of subgroups that has certain primitive prime divisors in \cite{bamberg2008overgroups}. Section \ref{sec_extField} is devoted to the extension field case and makes use of the bounds in Section \ref{sec:boundsection}, while Section \ref{sec:mainresults} handles the remaining cases. In Section \ref{sec:Conclud}, we conclude the paper with some open problems.

\section{Preliminaries}
\label{sec:Preli}
\subsection{Finite classical groups}
\label{subsec:FSG}
Let $\F_{q}$ be the finite field with $q$ elements, where $q=p^f$ with $p$ prime. Let $V$ be a vector space of dimension $d$ over $\F_q$ equipped with a reflexive sesquilinear form or quadratic form $\kappa$, where $\kappa$ is either nondegenerate or constantly zero. We use $\GamL(V)$ for the  semilinear transformation group of $V$ and use $\GL(V)$ for the linear transformation group. We will also write $\GamL_d(q)$, $\GL_d(q)$ instead if the context is clear.
\begin{table}[!h]\footnotesize\tabcolsep 16pt
	\centering
	\caption{The forms $\kappa$ on $V$}
\centering
	\begin{tabular}{cc}
		\toprule
		case $\bL$ & $\kappa$ is identically 0\\
		case $\bS$ & $\kappa=\f$, a nondegenerate symplectic form\\ 
		case $\bO^{\epsilon}$ & $\kappa=Q$, a nondegenerate quadratic form\\
		case $\bU$ & $\kappa=\f$, a nondegenerate unitary form\\
		\bottomrule	
	\end{tabular}
\end{table}
We take the same notation for classical groups as in \cite{kleidman1990subgroup}, except for that we regard the ambient space $V$ as a vector space over $\F_q$ in all cases. For instance, if $\kappa$ is a sesquilinear form, we define the $\kappa$-semisimilarity group as
\begin{equation}\label{eqn_Gam}
	 \Gamma(V,\kappa):=\{g\in\GamL(V):\,\kappa(xg,yg)=\lambda(g)\kappa(x,y)^{\sigma(g)}\textup{ for all }x\in V\},
\end{equation}
where $\sigma(g)\in\Gal(\F_{q}/\F_p)$, $\lambda(g)\in\F_{q}^{*}$. In the case where $\kappa$ is a quadratic form the group $\Gamma(V,\kappa)$ is defined similarly. We define the similarity group $\Delta(V,\kappa)$ as the kernel of $\sigma$, the isometry group  $I(V,\kappa)$ as $\ker(\sigma)\cap\ker(\lambda)$, and define $S(V,\kappa)$ to be the subgroup of $I(V,\kappa)$ consisting of elements of determinant $1$. In the cases $\bL,\bU,\bS$, we define $\Omega(V,\kappa)$ to be $S(V,\kappa)$; in the case $\bO$, we define $\Omega(V,\kappa)$ to be a certain subgroup of index $2$ in $S(V,\kappa)$. For more details about the group $\Omega(V,\kappa)$, please refer to  \cite[Section 2.5]{kleidman1990subgroup}.

We use the standard group theoretical notation as in \cite{bray2013maximal}. Let $A$ and $B$ be two groups. We use $A\times B$ to denote the direct product of $A$ and $B$, $A:B$ to denote a split extension, $A^{\cdot}B$ to denote a non-split extension (or possibly one in which $A$ is trivial), and use $A.B$ when we do not know or do not wish to specify whether the extension splits. For a finite group $G$, we use $G^{(\infty)}$ for the terminating member of the derived series of $G$. For a finite nonabelian simple group $L$, we have $|\Aut(L)|=|\textup{Out}(L)|\cdot|L|$, and the order of the outer automorphism group $\textup{Out}(L)$ is available from  Table 2.1.C, 2.1.D, 5.1.A, 5.1.B and 5.1.C of \cite{kleidman1990subgroup}.

\subsection{Intriguing sets}

Let $V$ be a vector space of dimension $d$ over $\F_q$ equipped with a nondegenerate reflexive sesquilinear form or quadratic form $\kappa$, and let $\cP$ be the associated polar space. A point of $\cP$ is defined as a $1$-dimensional totally isotropic/singular subspace of $V$. A maximal totally isotropic/singular subspace of $\cP$ is called a \textit{generator} of $\cP$. All generators have the same dimension $r$, which is called the \textit{rank} of $\cP$. A generator has $\frac{q^r-1}{q-1}$ points. An \textit{ovoid}  of $\cP$ is a set of points which intersects each generator in exactly one point. We denote by $\theta_r$ for the size of a putative ovoid, which we call the ovoid number. It holds that $|\cP|=\theta_r\cdot \frac{q^r-1}{q-1}$. We list the ranks and the ovoid numbers of the five classical polar spaces in Table \ref{tab_valuede}. In the last column of the table, we list the smallest integer $e$ such that $\theta_r|q^{e}-1$.
\begin{table}[!h]\footnotesize\tabcolsep 16pt
	\centering
	\caption{The parameters $r$, $\theta_{r}$ and the smallest $e$ such that $\theta_r|q^{e}-1$}	\label{tab_valuede}
	\begin{tabular}{ccccccc}
		\toprule
		& $d$ & $f$& polar space $\cP$ & rank $r$ & ovoid number $\theta_{r}$ & $e$ \\
		\hline
		$\bS$ & even&-& $W(d-1,q)$ & $d/2$ & $q^{d/2}+1$ &$d$\\
		\hline
		\multirow{3}*{$\bO$}
		&even &-& $Q^+(d-1,q)$ & $d/2$ & $q^{d/2-1}+1$ &$d-2$\\
		& even &-& $Q^-(d-1,q)$ & $d/2-1$ & $q^{d/2}+1$ &$d$\\
		& odd &-& $Q(d-1,q)$ & $(d-1)/2$ & $q^{(d-1)/2}+1$ & $d-1$\\
		\hline
		\multirow{2}*{$\bU$}
		& odd &even &$H(d-1,q)$ & $(d-1)/2$ & $q^{d/2}+1$ &$d$\\
		& even &even& $H(d-1,q)$ & $d/2$ & $q^{(d-1)/2}+1$ &$d-1$\\
		\bottomrule
	\end{tabular}
\end{table}	

Suppose that $r\ge 2$, and let $\cM$ be a nonempty set of  $\cP$. It is an \textit{intriguing set} if there exist some constants $h_1,h_2$ such that $|P^{\perp}\cap\cM|=h_{1}$ or $h_{2}$ according as $P\in\cM$ or not, where $P$ ranges over the points of $\cP$, cf. \cite{BambergTightsets2007}. An intriguing set $\cM$ is proper if $\cM\ne\cP$. There are two types of intriguing sets:
\begin{enumerate}
	\item[(1)]$i$-tight sets: $|\cM|=i\cdot\frac{q^r-1}{q-1}$, $h_1=q^{r-1}+i\cdot\frac{q^{r-1}-1}{q-1}$, $h_2=i\cdot\frac{q^{r-1}-1}{q-1}$, and
	\item[(2)]$m$-ovoids: $|\cM|=m\theta_r$, $h_1=(m-1)\theta_{r-1}+1$, $h_2=m\theta_{r-1}$.
\end{enumerate}
We refer to \cite{BambergTightsets2007} for more properties of intriguing sets. In particular, if $H$ is a subgroup of $\Gamma(V,\kappa)$ that has exactly two orbits $O_1,\,O_2$ on the points of $\cP$, then both $O_1$ and $O_2$ are intriguing sets of the same type.

\begin{remark}\label{rem_2Fact}
We shall use the following facts from \cite{BambergTightsets2007}.
\begin{enumerate}
\item[(1)] Suppose that $\cM_1,\cM_2$ are $m_1$- and $m_2$-ovoids in a classical polar space $\cP$ such that $\cM_1\subseteq\cM_2$. It is routine to check that $\cM_2\setminus \cM_1$ is a $(m_2-m_1)$-ovoid, cf. \cite[Section 2]{BambergTightsets2007}.
\item[(2)] Suppose that both $q$ and $d$ are even, and let $\cP=Q^-(d-1,q)$ be a polar space with ambient space $V$. The associated bilinear form is alternating and nondegenerate, so it defines a symplectic polar space $W(d-1,q)$. By \cite[Theorem 11]{BambergTightsets2007}, an $m$-ovoid of $Q^{-}(d-1,q)$ is also an $m$-ovoid in $W(d-1,q)$.
\end{enumerate}
\end{remark}

\subsection{The field reduction}
Let $b$ be a  divisor of an integer $d$ with $b>1$. Suppose that $V'$ is a $d/b$-dimensional vector space over the finite field $\F_{q^b}$, equipped with a non-degenerate sesquilinear or quadratic form $\kappa'$. Let $\cP'$ be the associated polar space. Let $\tr_{q^b/q}$ be the relative trace function from the field $\F_{q^b}$ to $\F_{q}$.  We regard $V'$ as a $d$-dimensional vector space $V$ over $\F_q$. By composing $\kappa'$ and the trace function, we obtain a new form $\kappa$ on $V$. Let $\cP$ be the polar space defined by $\kappa$.  We define
\begin{equation}\label{eqn_GamV'def}
\Gamma^{\#}(V',\kappa'):=\{g\in\Gamma(V',\kappa')|\,\lambda(g)\in\F_{q}^{*}\},
\end{equation}
which is a subgroup of $\Gamma(V,\kappa)$. In Table \ref{tab_extfieldP'gt0}, we list all the cases that can be obtained by iteratively composing the field reductions in \cite[Table 4.3A]{kleidman1990subgroup}. The first half are those cases where $\cP$ and $\cP'$ are of the same type, and the second half are the cases where they are of different types. The table also lists the sizes of $\cP$ and $\cP'$ for reference.
For more details, please refer to \cite{gill2006polar}.

\begin{table}[!ht]
	\centering
	\caption{The field reduction}
	\label{tab_extfieldP'gt0}
	\scalebox{0.7}{
		\begin{tabular}{ccccccc}
			\toprule
			Case & $\cP'$ & $|\cP'|$ & $\cP$ & $|\cP|$ &$\kappa$& Condition\\ \hline
			1 &$W(d/b-1,q^b)$ & $\frac{q^d-1}{q^b-1}$ & $W(d-1,q)$ & $\frac{q^d-1}{q-1}$ & $\tr_{q^b/q}\circ\kappa'$&$d/b$ even\\
			2 & $Q^{+}(d/b-1,q^b)$ & $(q^{d/2-b}+1)\frac{q^{d/2}-1}{q^b-1}$ & $Q^{+}(d-1,q)$ & $(q^{d/2-1}+1)\frac{q^{d/2}-1}{q-1}$ & $\tr_{q^b/q}\circ\kappa'$& $d/b$ even \\
			3 & $Q^{-}(d/b-1,q^b)$ & $(q^{d/2}+1)\frac{q^{d/2-b}-1}{q^b-1}$ & $Q^{-}(d-1,q)$ & $(q^{d/2}+1)\frac{q^{d/2-1}-1}{q-1}$ & $\tr_{q^b/q}\circ\kappa'$& $d/b$ even \\
			4 & $Q(d/b-1,q^b)$ & $\frac{q^{d-b}-1}{q^b-1}$ & $Q(d-1,q)$ & $\frac{q^{d-1}-1}{q-1}$ & $\tr_{q^b/q}\circ\kappa'$& $dq$ odd \\
			5 & $H(d/b-1,q^b)$ & $(q^{d/2}+1)\frac{q^{(d-b)/2}-1}{q^b-1}$ & $H(d-1,q)$ & $(q^{d/2}+1)\frac{q^{(d-1)/2}-1}{q-1}$ & $\tr_{q^b/q}\circ\kappa'$& $d$ odd \\
			6 & $H(d/b-1,q^b)$ & $(q^{(d-b)/2}+1)\frac{q^{d/2}-1}{q^b-1}$ & $H(d-1,q)$ & $(q^{(d-1)/2}+1)\frac{q^{d/2}-1}{q-1}$ & $\tr_{q^b/q}\circ\kappa'$& $d/b$ even, $b$ odd \\
			\hline
			7 & $H(d/b-1,q^b)$ & $(q^{d/2}+1)\frac{q^{(d-b)/2}-1}{q^b-1}$ & $W(d-1,q)$ & $\frac{q^{d}-1}{q-1}$ & $\tr_{q^b/q}\circ\lambda\kappa'$&  $d/b$ odd, \\
			&&&&&&$\lambda\in\F_{q^b}$ s.t. $\lambda^{q^{b/2}}+\lambda=0$\\
			8 & $H(d/b-1,q^b)$ & $(q^{(d-b)/2}+1)\frac{q^{d/2}-1}{q^b-1}$ & $W(d-1,q)$ & $\frac{q^{d}-1}{q-1}$ & $\tr_{q^b/q}\circ\lambda\kappa'$&$d/b$ even,  \\
			&&&&&&$\lambda\in\F_{q^b}$ s.t.$\lambda^{q^{b/2}}+\lambda=0$\\
			9 & $H(d/b-1,q^b)$ & $(q^{d/2}+1)\frac{q^{(d-b)/2}-1}{q^b-1}$ & $Q^{-}(d-1,q)$ & $(q^{d/2}+1)\frac{q^{d/2-1}-1}{q-1}$ & $\tr_{q^{b/2}/q}\circ\kappa'(-,-)$&$d/b$ odd, $b$ even \\
			10& $H(d/b-1,q^b)$ & $(q^{(d-b)/2}+1)\frac{q^{d/2}-1}{q^b-1}$ & $Q^{+}(d-1,q)$ & $(q^{d/2-1}+1)\frac{q^{d/2}-1}{q-1}$ & $\tr_{q^{b/2}/q}\circ\kappa'(-,-)$&$d/b$ even, $b$ even  \\
			11 & $Q(d/b-1,q^b)$ & $\frac{q^{d-b}-1}{q^b-1}$ & $Q^{+}(d-1,q)$ & $(q^{d/2-1}+1)\frac{q^{d/2}-1}{q-1}$ & $\tr_{q^b/q}\circ\lambda\kappa'$&$qd/b$ odd, $\lambda\in\F_{q^b}^{*}$\\
			12 & $Q(d/b-1,q^b)$ & $\frac{q^{d-b}-1}{q^b-1}$ & $Q^{-}(d-1,q)$ & $(q^{d/2}+1)\frac{q^{d/2-1}-1}{q-1}$ & $\tr_{q^b/q}\circ\lambda\kappa'$&$qd/b$ odd, $\lambda\in\F_{q^b}^{*}$ \\
			\bottomrule
	\end{tabular}}
\end{table}

\subsection{The fully deleted modules of $A_n$}\label{subsec_FDM}
Suppose that $p$ is a prime and $n\geq 5$. The alternating group $A_{n}$ acts on the vector space $\F_{p}^{n}$ by permuting the coordinates. Set
\begin{equation}\label{eqn_UW}
	U=\left\{(a_{1},\ldots,a_{n}):\,\sum_{i=1}^{n}a_{i}=0\right\},\;
	W=\{(a,\ldots,a):\,a\in\F_{p}\},
\end{equation}
and define $V:=U/(U\cap W)$. The dimension of $V$ is $n-2$ or $n-1$ according as $p$ divides $n$ or not. Both $U$ and $W$ are $A_n$-invariant, and so there is an induced action of $A_n$ on $V$. We call $V$ the \textit{fully deleted permutation module} of $A_n$. The action of $A_n$ on $V$ gives an embedding of $A_n$ as a  class $\mathscr{S}$ subgroup of $\Gamma(V,\kappa)$, cf. \cite[Definition 2.1.3]{bray2013maximal} , and it preserves the following symmetric bilinear form on $U$:
\begin{equation}
	\label{eq_BAn}
	B((a_{1},\ldots,a_{n}),(b_{1},\ldots,b_{n}))=\sum_{i=1}^{n}a_{i}b_{i}.
\end{equation}
It induces a nondegenerate symmetric bilinear form on $V$, which we also write as $B$ by abuse of notation.  For $p$ odd, the form $B$ on $V$ yields a nondegenerate $A_n$-invariant quadratic form $Q(a)=2^{-1}B(a,a)$, where $a\in V$. This leads to the following group embeddings:
	\begin{align*}
	A_{2m+1}&\leq\left\{\begin{array}{llll}
		\Omega_{2m}^{+}(p), & \textup{ if } p\nmid 2m+1,\, (-1)^m\cdot n   \text{ is a square of $\F_p$};\\
		\Omega_{2m}^{-}(p), & \textup{ if } p\nmid 2m+1,\, (-1)^m\cdot n  \text{ is a nonsquare of $\F_p$};\\
		\Omega_{2m-1}(p), & \textup{ if } p\mid 2m+1;\\
	\end{array}\right.\\
	A_{2m+2}&\leq\left\{\begin{array}{llll}
		\Omega_{2m+1}(p), & \textup{ if } p\nmid 2m+2;\\
		\Omega_{2m}^{+}(p), & \textup{ if } p\mid 2m+2,  \text{ $m$ is odd  or $m$ is even},  p \equiv 1 \text{(mod 4)}; \\
		\Omega_{2m}^{-}(p), & \textup{ if } p\mid 2m+2, \text{ $m$ is even}, p \equiv 3 \text{(mod 4)}. \\
	\end{array}\right.
\end{align*}

Now let $p=2$ and define $Q:\,U\rightarrow\F_{2}$ by
\begin{align}
	\label{eq_QAn}
	Q(a)=\begin{cases}1, & \textup{ if }\omega(a)\equiv 2(\textup{mod }4),\\
		0, & \textup{ if }\omega(a)\equiv 0(\textup{mod }4),                             \end{cases}
\end{align}
where $\omega(a)=\#\{i:\,a_i\ne 0\}$.  The form $Q$ is $A_{n}$-invariant and its associated bilinear form is $B$. If $n\not\equiv2\pmod{4}$, then $Q$ induces a nondegenerate quadratic form on $V$ which we again write as $Q$. If $n\equiv 2\pmod{4}$, then $Q$ is degenerate and there is no nondegenerate $A_n$-invariant quadratic form on $V$. By \cite[p.187]{kleidman1990subgroup}, this yields the following group embeddings
\begin{equation}
	\label{eqn_A2leq}
	A_{2m+2}\leq\left\{\begin{array}{llll}
		\Omega_{2m}^{+}(2), & \textup{ if } m\equiv 3(\textup{mod }4),\\
		\Omega_{2m}^{-}(2), & \textup{ if } m\equiv 1(\textup{mod }4),\\
		\Sp_{2m}(2), & \textup{ if } m\textup{ is even.}\\
	\end{array}\right.
\end{equation}

\begin{equation}
	\label{eqn_A1leq}
	A_{2m+1}\leq\left\{\begin{array}{llll}
		\Omega_{2m}^{+}(2), & \textup{ if } m\equiv 0(\textup{mod }4),\\
		\Omega_{2m}^{-}(2), & \textup{ if } m\equiv 2(\textup{mod }4),\\
		\Omega_{2m}^{\mp}(2), & \textup{ if }m\equiv \pm1(\textup{mod }4).\\
	\end{array}\right.
\end{equation}

\subsection{Frobenius-Schur indicator}
Suppose that $G$ is a finite group, $F$ is a field, and $M$ is a finitely generated, absolutely irreducible $FG$-module. The \textit{Frobenius-Schur indicator} of $M$ is an integer in the set $\{+1,\,0,\,-1\}$. It is defined to be $0$ if and only if $M$ is not self-dual. The indicator is $+1$ if $M$ carries a nondegenerate $G$-invariant quadratic form, and it is $-1$ if $M$ is self-dual and carries a nondegenerate $G$-invariant alternating form but no $G$-invariant quadratic form. There can be at most one nondegenerate $G$-invariant reflexive sesquilinear form on the module $M$ up to scalar multiplication. As in \cite{Hissrepreofquasi2001}, we use $\circ,\,+,\,-$ to denote $0,\,+1,\,-1$ respectively. In the case the indicator is $\circ$, we need the following technical result.
\begin{lemma}[\cite{bray2013maximal}, Corollary 4.4.2]
	\label{lem_FScircDet}
	Suppose that the character ring of an absolutely irreducible representation with indicator $\circ$ of a group $G$ over $\F_{q^2}$ is generated over $\Z$  by the quadratic irrationalities $a_1,\cdots,a_r$, and let $\bar{a}_i$ denote a $p$-modular reduction of $a_i$ to $\F_{q^2}$. Then the image of $G$ under the representation consists of isometries of a unitary form if and only if $a_i\in\mathbb{R}\leftrightarrow\bar{a}_i\in\F_q$ for $1\le i\le r$.
\end{lemma}
For more details on how to determine the Frobenius-Schur indicators, please refer to \cite[p.725]{curtis1987methods} and \cite{Hissrepreofquasi2001}. In the natural characteristic case, we use the data in \cite{LubeckSmalldegree}, and in the other cases we consult \cite{Hissrepreofquasi2001}.

\subsection{Primitive divisor}\label{sec_primdiv}

Let $n,\,k$ be positive integers. A divisor $t$ of $n^k-1$ is called \textit{primitive} if $t$ is coprime to each $n^i-1$ for $1\le i\le k-1$, and we call the largest primitive divisor $\Phi_{k}^{*}(n)$ of $n^k-1$ \textit{the primitive part}. In particular, a prime divisor $t$ of $n^k-1$ is called a primitive prime divisor if $n$ has order $k$ modulo $t$. As a corollary, we have $t\equiv 1\pmod{k}$, and in particular $\gcd(t,k)=1$. If $n^k-1$ has a primitive prime divisor, then $\Phi_{k}^{*}(n)>1$ and it is relatively prime to $k$.
\begin{lemma}[\cite{zsigmondy1892theorie}]
	\label{lem_zsigmondy}
	Let $n$ and $k$ be positive integers such that $n>1$, $k\geq 3$ and $(n,k)\neq(2,6)$. Then $n^k-1$ has at least one primitive prime divisor.
\end{lemma}

The primitive prime divisors play an important role in the study of geometric objects that satisfy certain transitivity conditions, see \cite{Bamberg2009Classificationovoid, FengLi} for instance.
\medskip

\section{The main results}
\label{sec:MainResults}

Let $V$ be a vector space of dimension $d$ over $\F_q$ equipped with a nondegenerate reflexive sesquilinear or quadratic form $\kappa$, and suppose that $d-1$ and $q$ are not both even in the case $\bO$. Let $\cP$ be the associated classical polar space. We first list some examples of transitive $m$-ovoids of $\cP$. In the sequel, we use $\GO^{\epsilon}_{d}(q)$ and $\textup{O}^{\epsilon}$ to denote the similarity and isometry group respectively when $\kappa$ is a quadratic form of type $\epsilon$ on $V$.

\begin{example}[Classical examples]\label{exa_Classical} Suppose that both $q$ and $d$ are even and $\kappa$ is an alternating form. Let $Q^{-}(d-1,q)$ be an elliptic quadric whose associated polar form is $\kappa$. Then the points of $Q^{-}(d-1,q)$ and its complement form a $\frac{q^{d/2-1}-1}{q-1}$- and $q^{d/2-1}$- ovoid of $\cP=W(d-1,q)$ respectively. Here $\Omega^{-}_{d}(q)$ acts transitively on both $Q^{-}(d-1,q)$ and its complement in $W(d-1,q)$.
\end{example}	

\begin{example}	[Imprimitive examples]\label{exa_Imprimitive}
	Suppose that $q=3$, $d=5,7$ or $8$, and $Q(\sum_{i=1}^{d}x_{i}e_{i})=\sum_{i=1}^{d}x_{i}^2$, where the $e_{i}$'s form a basis of $V$. Let $\GO_{1}(3)\wr S_{d}$ be the stabilizer of the direct sum decomposition $V=U_{1}\oplus U_{2}\oplus\ldots\oplus U_{d}$, where $U_{i}=\F_{3}\cdot e_{i}$. By examining the irreducible subgroups of $\GO_{1}(3)\wr S_{d}$ one by one using Magma \cite{Magma}, we obtain transitive $2$-ovoids of $Q(d-1,3)$ in the case $d=5$, transitive $1,\,2,\,3,\,4,\,5,\,6,\,8$-ovoids of $Q(d-1,3)$ in the case $d=7$, and transitive $8,\,16,\,24,\,32$-ovoids of $Q^{+}(d-1,3)$ in the case $d=8$.
\end{example}

\begin{example}[Extension field examples]\label{exa_Extension}
	Let $b$ be a proper divisor of $d$. Suppose that there is a $d/b$-dimensional vector space $V'$ over $\F_{q^b}$ equipped with a form $\kappa'$ such that $(V',\kappa',\cP')$ and $(V,\kappa,\cP)$ are associated as in Table \ref{tab_extfieldP'gt0}. Set
	\begin{equation}\label{eqn_M1def}
		\cM_{1}=\{\left\langle \lambda v\right\rangle_{\F_{q}}:\,\lambda\in\F_{q^b}^{*},\, \left\langle v\right\rangle_{\F_{q^b}}\in\cP'\}.
	\end{equation}
	It is an $m$-ovoid of $\cP$ in the cases listed below in Table \ref{tab_extfdexamp}, cf. \cite[Table 1]{KellyCons}. The lower bound on $d/b$ is to make sure that $\cP'$ is not empty. The group $\Omega(V',\kappa')$ is transitive on $\cM_{1}$ by \cite[Lemma 2.10.5]{kleidman1990subgroup}.
\end{example}	
\begin{table}[H]
	\caption{Transitive $m$-ovoids in the extension field case}
	\label{tab_extfdexamp}
	\centering
	\begin{tabular}{cccc}
		\toprule
		$\cP$ &  $\cP'$ &  $m$ & Condition\\  \hline
$H(d-1,q)$ & $H(d/b-1,q^b)$ & $\frac{q^{(d-b)/2}-1}{q-1}$ &  $d$ odd, $d/b\ge 3$ \\
$Q^{-}(d-1,q)$ & $Q^{-}(d/b-1,q^b)$ &  $\frac{q^{d/2-b}-1}{q-1}$ & $d/b\ge 4$ even \\
$Q^{-}(d-1,q)$ & $H(d/b-1,q^b)$ &$\frac{q^{(d-b)/2}-1}{q-1}$ &$b$ even, $d/b\ge 3$ odd\\
$Q^{+}(d-1,q)$ & $Q(d/2-1,q^2)$ & $\frac{q^{d/2-1}-1}{q-1}$ & $d/2\ge 3$ odd \\
$W(d-1,q)$ & $H(d/b-1,q^b)$ & $\frac{q^{(d-b)/2}-1}{q-1}$ & $b$ even, $d/b\ge 3$ odd\\
		\bottomrule
	\end{tabular}
\end{table}

\begin{example}[Symplectic type examples]	
	\label{exa_Symplectic}
	The extraspecial $2$-group $2_{-}^{1+4}$ has an irreducible module $V$ of dimension 4 over $\F_{3}$, and it preserves a nondegenerate alternating form $\kappa$ on $V$. The normalizer of $2_{-}^{1+4}$ in $\Omega(V,\kappa)$ is $2_{-}^{1+4}.O_{4}^{-}(2)$. By examining the orbits of its subgroups, we obtain two transitive $2$-ovoids of $W(3,3)$ with stabilizer $2.\Omega_4^-(2)$. The latter group has exactly two orbits on $W(3,3)$.
\end{example}

In the following examples, the $m$-ovoids admit a transitive automorphism group $H_0$ such that $S\le H_0\le M$, where $M$ is an almost simple subgroup of $\textup{P}\Gamma(V,\kappa)$ with socle $S$. Let $H$ be the full preimage of $H_0$ in $\Gamma(V,\kappa)$. The group $\tilde{S}:=H^{(\infty)}$ acts on $V$ absolutely irreducibly in all those examples.

\begin{example}[Alternating group]\label{exa_Alternating}
Take $S=A_{n}$ for some $n\geq5$, and let $(q,d,e)$ be as in Table \ref{tab_valuede}.
\begin{enumerate}
\item[(1)] Let $V$ be the fully deleted module for $A_n$. We check with Magma \cite{Magma} that a single $A_n$-orbit forms an $m$-ovoid in the cases listed in Table \ref{tab_fully}. In each case, the $A_{n}$-orbits have different sizes, so they are all stabilized by $M$.

\begin{table}[!h]
	\caption{Transitive $m$-ovoids from the fully deleted modules of $A_n$}
	\label{tab_fully}
	\centering
	\begin{threeparttable}
		\begin{tabular}{ccccc}
			\toprule
			$\cP$  & $A_n$ & $\#$ $A_n$-orbits& $m$ &$(q,\,d,\,e)$ \\
			\hline
			$Q(6,3)$    & $A_{8}$  & $3$ & $1,\,2,\,10$ &$(3,\,7,\,6)$\\
			            & $A_9$    & $2$ & $3,10$       &$(3,\,7,\,6)$\\
			$Q^+(7,5)$  & $A_{9}$  & $8$ & $1,4,6,10,20,30,40,45$ &$(5,\,8,\,6)$ \\
			            & $A_{10}$ & $4$ & $1,10,45,100$ &$(5,\,8,\,6)$\\
			$Q^+(13,2)$ & $A_{15}$ & $3$ & $7,\,21,\,99$ & $(2,\,14,\,12)$\\
			            & $A_{16}$ & $2$ & $28,\,99$ & $(2,\,14,\,12)$\\
			\bottomrule	
		\end{tabular}
	\end{threeparttable}
\end{table}	

\item[(2)] We then consider the other absolutely irreducible representations of $A_n$ or its covering groups. We check with Magma \cite{Magma} that a single $A_n$-orbit forms an $m$-ovoid in the cases listed in Table  \ref{tab_otherAlt}. Those transitive $m$-ovoids are also $M$-invariant.
\begin{table}[!h]		
	\caption{Transitive $m$-ovoids from other representations of $A_{n}$ and its covering groups}
	\label{tab_otherAlt}
	\centering
	\begin{threeparttable}
		\begin{tabular}{ccccc} \toprule
			$\cP$ & $\tilde{S}$& $\#$ $A_n$-orbits & $m$ &$(q,\,d,\,e)$ \\ \hline
			$Q^+(7,5)$ & $2^\cdot A_{10}$&  $2$ &   $36,\,120$ & $(5,\,8,\,6)$ \\	
			           & $2^\cdot A_{9}$ &  $3$ &  $36,\,60,\,60$  &$(5,\,8,\,6)$ \\
			           & $A_{7}$         &  $16$&  $1,\,5,\,10,\,20$ &  $(5,\,8,\,6)$ \\
			\bottomrule
		\end{tabular}	
	\end{threeparttable}
\end{table}
\end{enumerate}
\end{example}

\begin{example}[Sporadic case]\label{exa_thmsporadic}	
	The group $2^\cdot J_{2}$ has a unique  $6$-dimensional absolutely irreducible module $V$ over $\F_{5}$, cf. \cite{Lindsey1970Onsix}. It preserves a nondegenerate alternating form $\kappa$. We verify by Magma \cite{Magma} that $S=J_{2}$ has two orbits on $W(5,5)$, and they form transitive $15$- and $16$-ovoids of $W(5,5)$ respectively.
\end{example}

\begin{example}[Cross-characteristic case]\label{exa_thmCross}
Let $S$ be a simple group of Lie type whose defining characteristic is different from $p$,
and let $(q,d,e)$ be as in Table \ref{tab_valuede}. We check with Magma that we obtain transitive $m$-ovoids from a single $S$-orbit or $M$-orbit in the cases listed in Table \ref{tab_thmcross}.  In the column 4 of Table \ref{tab_thmcross}, a parameter in bold inside a bracket indicates that the $m$-ovoid is $M$-transitive but not $S$-transitive. All the other $m$-ovoids consist of a single $S$-orbit.

\begin{table}[!ht]
	\centering   \caption{Transitive $m$-ovoids in the cross-characteristic case}\label{tab_thmcross}
	\begin{threeparttable}	
		\begin{tabular}{ccccc} \toprule
			$\cP$ & $\tilde{S}$ & $\#$ $S$-orbits  & $m$ & $(q,\,d,\,e)$\\ \hline
			$Q(6,3)$   & $\PSL_{2}(7)$& $6$  &  $1,\,2$ & $(3,\,7,\,6)$\\
			           & $\PSL_2(8)$  & $5$  &   $1,\,3,\,\{\bf 9\}$ & $(3,\,7,\,6)$ \\
                       & $\PSp_{6}(2)$& $2$  &    $1,\,12$  & $(3,\,7,\,6)$ \\
			$Q(6,5)$   & $\PSL_2(8)$  & $13$ &   $\{\bf 6\}$ & $(5,\,7,\,6)$ \\	
			           & $\PSU_3(3)$  & $3$  &    $16$ & $(5,\,7,\,6)$ \\	
			           & $\PSp_{6}(2)$  & $2$ &    $15,\,16$ & $(5,\,7,\,6)$ \\
			$Q^-(5,3)$ & $2^\cdot \PSL_{3}(4)$ & $2$ &   $2$ & $(3,\,6,\,6)$ \\
			$Q^+(7,5)$ & $\PSL(2,8)$    & 48 &  $1,4,\{\bf 6\}$ & $(5,8,6)$ \\
			           & $2^\cdot\PSp_{6}(2)$	&$2$  &   $60,\,96$  & $(5,\,8,\,6)$ \\
			           & $2^{\cdot}\POmeg^{+}_{8}(2)$ & $2$ & $60,\, 96$ & $(5,\,8,\,6)$ \\
			$W(7,2)$   & $\PSL_2(17)$ & $2$  &  $6,\,9$ & $(2,\,8,\,8)$ \\
			$W(11,2)$  & $\PSL_2(25)$ & $6$  & $12$ & $(2,\,12,\,12)$ \\
			\bottomrule
		\end{tabular}			
\begin{tablenotes}
\footnotesize\item[*] ``$\{\bf m\}$" indicates that  the $m$-ovoid is $M$-transitive but not $S$-transitive.
		\end{tablenotes}
	\end{threeparttable}
\end{table}
\end{example}

\begin{remark}
Here are some details for the case $(\cP,\tilde{S})=\left(Q^{+}(7,5),2^{\cdot}\POmeg_{8}^{+}(2)\right)$ in Example \ref{exa_thmCross}. Let $H=2.\textup{GO}^+_8(2)$ be the Weyl group of the $E_8$-root lattice in $\mathbb{R}^8$, generated by the $120$ reflections in the $240$ root vectors, cf. \cite{Atlas}. It preserves the Euclidean inner product $(x,y)=\sum_{i=1}^8x_iy_i$. This representation of $H$ has integer coefficients, so $\otimes\F_5$ yields an embedding of $H$ in $\textup{GO}^+_8(5)$. It has exactly two orbits on $Q^+(7,5)$. The orbit that contains $\la(1,3,0,\cdots,0)\ra$ is a $60$-ovoid, and the other orbit that contains $\la(1,1,1,1,1,0,0,0)\ra$ is a $96$-ovoid.
\end{remark}

\begin{example}[Natural-characteristic case]
	\label{exa_thmNatural}
Let $S$ be a simple group of Lie type of characteristic $p$. In Table \ref{tab_thmNatural}, we list five cases where a single $S$-orbit forms an $m$-ovoid. In the column 4, a parameter in bold inside a bracket indicates that the $m$-ovoid is $M$-transitive but not $S$-transitive. We give some details below.
\begin{enumerate}
\item[(a)] For case 1, there are three $S$-orbits and the shortest one is the Ree ovoid. By taking the same matrix description of  $^2G_{2}(q)$ as in \cite{BaarnhielmRecognising} and \cite{KemperMatrixgenerators}, we can show that the other two orbits are also $q$- and $q^2$-orbits respectively.
\item[(b)] For cases 2 and 3, there are respectively two $S$-orbits one of which is an ovoid, cf. \cite[Sections 4, 7]{KantorOvoids}.
\item[(c)] For case 4, there are three $\PGU_{3}(q)$-orbits $\mathcal{O}_{1}$, $\mathcal{O}_{2}$ and $\mathcal{O}_{3}$ with length $q^3+1$, $(q+q^2)(q^3+1)$ and $q^3(q^3+1)$ respectively. The orbit $\mathcal{O}_{1}$ is the unitary ovoid,  and $\mathcal{O}_{2}$, $\mathcal{O}_{3}$ are $(q+q^2)$- and $q^3$-ovoids respectively, cf. \cite[Section 4]{KantorOvoids}. The subgroup $\PSU_{3}(q)$ is transitive on both $\mathcal{O}_{1}$ and $\mathcal{O}_{3}$, and splits $\mathcal{O}_{2}$ into three orbits. There is no $\PSU_{3}(q)$-invariant $m$-ovoid properly contained in $O_2$, cf. \cite{FengLT}.
\item[(d)] For case 5, there are two $S$-orbits and the shortest one is an ovoid, cf. \cite[Theorem 27.3]{LuneburgTranslation} and \cite{SegreCaps}.
\end{enumerate}
\begin{table}[!h]
		\centering
	\begin{threeparttable}
		\caption{Transitive $m$-ovoids in the natural characteristic case}
		\label{tab_thmNatural}	
		\begin{tabular}{cccccc}
			\toprule
			Case&$\cP$   &  $S=H^{(\infty)}$ &$\#\,S$-\text{orbits} & $m$& Condition\\ \hline
			1&	$Q(6,q)$ & ${^{2}G}_{2}(q)$ &3& $1,\,q,\,q^2$ &$p=3$, $f\geq3$ odd\\	
			2&	$Q(6,q)$ & $\PSU_{3}(q)$ &2& $1,\,q+q^2$&	$p=3$\\
			3&  $Q^{+}(7,q)$ & $\PSL_{2}(q^3)$ &2& $1,\,q+q^2+q^3$ &	$p=2$ \\
			4&	$Q^{+}(7,q)$ & $\PSU_{3}(q)$ & $5$ &  $1,\,\{\mathbf{q+q^2}\},\,q^3$ &	$q\equiv2(\textup{mod }3)$\\
			5&  $W(3,q)$ & $\text{Sz}(q)$ &2& $1,\,q$ &	$p=2$, $f\geq3$ odd\\
			\bottomrule
		\end{tabular}
\begin{tablenotes}
\footnotesize\item[*] ``$\{\bf m\}$" indicates that  the $m$-ovoid is $M$-transitive but not $S$-transitive.
\end{tablenotes}
	\end{threeparttable}
\end{table}
\end{example}

\begin{remark}
For case (a) of Example \ref{exa_thmNatural}, the Ree group $^{2}{G}_2(q)$ fixes the Ree spread of the split Cayley hexagon $\mathcal{H}(q)$, cf. \cite[Table 1]{BambergclassificationTransi}. It is a $1$-system of $Q(6,q)$ and hence the set of points they cover forms a $(q+1)$-ovoid $\cM$ of $Q(6,q)$, cf. \cite{Shultmsystem}. The set $\cM$ is exactly the union of the two short orbits.
\end{remark}

In the next example, we handle the case $\cP=Q^+(7,2)$.
\begin{example}\label{exa_Q+72}
	By Magma \cite{Magma}, we are able to enumerate all the transitive $m$-ovoids of $Q^+(7,2)$. There are transitive $m$-ovoids for $m\in\{1,2,3,4,6,8,9,12,14\}$. In particular, the $1$-ovoid is the Dye ovoid and the $14$-ovoid is its complement, cf. \cite{Bamberg2009Classificationovoid,Dye1977partitions}.
\end{example}

Finally, we present some more examples that arise from field reductions.
\begin{example}\label{exa_FR_PSU3_QPlus7}
	Suppose that $q=3^f$, $\cP'=Q(6,q^2)$, $\cP=Q^+(13,q)$. From \cite[Section 4]{KantorOvoids} we deduce that $\PSU_{3}(q)$ has exactly two orbits on the nonzero singular vectors of $V'$ and has exactly one orbit on the nonzero vectors $v\in V'$ such that  $v^\perp\cap Q(6,q^2)=Q^+(5,q^2)$. If $v$ is such a vector, then we can show that $\la \mu v\ra_{\F_q}\in Q^+(13,q)$ for some $\mu\in\F_{q^2}^*$ by using the explicit model in \cite{KantorOvoids}. We conclude that $\PSU_{3}(q)$ has three orbits on $Q^+(13,q)$. Since one orbit corresponds to an ovoid of $Q(6,q^2)$ by \cite{KantorOvoids}, we deduce that the three orbits on $Q^+(13,q)$ are $(q+1)$-, $q^2(q^2+1)(q+1)$- and $q^6$-ovoids of $Q^+(13,q)$ respectively by Example \ref{exa_Extension} and \cite{KellyCons}.
\end{example}

\begin{example}\label{exa_FR_W_Sz}
Suppose that $q=2^f$ with $f$ odd and $b>1$ is an odd integer. The group $\textup{Sz}(q^b)$ has two orbits on $W(3,q^b)$ one of which is an ovoid, cf. \cite[Theorem 27.3]{LuneburgTranslation}. By the matrix presentations of the Suzuki groups in \cite{SuzukiPNAS} and the same argument as in the proof of \cite[Lemma 2.11]{LiebeckRank3}, we deduce that $\textup{Sz}(q^b)$ has two orbits on the nonzero vectors of the ambient space. This leads to transitive $\frac{q^b-1}{q-1}$- and $q^b\frac{q^b-1}{q-1}$-ovoids of $W(4b-1,q)$ respectively by field reduction, cf. \cite{KellyCons}.
\end{example}

\begin{example}\label{exa_FR_M1compl}
	Take the same notation as in Example \ref{exa_Extension}. Then $\cP\setminus\cM_1$ is an $m$-ovoid with $\Omega(V',\kappa')$ acting transitively on it in the following cases:
\begin{enumerate}
\item[(1)] $d/2$ even, $\cP=Q^-(d-1,q)$, $\cP'=Q^-(d/2-1,q^2)$, $m=q^{d/2-2}$, cf. Section \ref{subsec_II_QM};
\item[(1')] $d/4\ge 3$ odd, $\cP=Q^-(d-1,q)$, $\cP'=H(d/4-1,q^4)$, $m=q^{d/2-2}$, cf. Section \ref{subsec_II_QM};
\item[(2)] $d/2\ge 3$ odd, $\cP=W(d-1,q)$, $\cP'=H(d/2-1,q^2)$, $m=q^{d/2-1}$, cf. Section \ref{subsec_II_W};
\item[(3)] $qd/2$ odd, $\cP=Q^+(d-1,q)$, $\cP'=Q(d/2-1,q^2)$, $m=q^{d/2-1}$, cf. Section \ref{subsec_II_QP}.
\end{enumerate}
In the case $q$ is odd,  $\cP=Q^+(13,q)$ and $\cP'=Q(6,q^2)$, $\cP\setminus\cM_1$ (resp. $\cM_{1}$) is a transitive $q^6$  (resp. $\frac{q^6-1}{q-1}$)-ovoid with $G_{2}(q^2)$ as a transitive automorphism group by Example \ref{exa_Extension} and \cite[Theorem 7.1]{giudici2020subgroups}.
\end{example}

\begin{remark}
Take the same notation as in Example \ref{exa_Extension}.
\begin{enumerate}
\item[(1)]When $d/4$ is an odd integer greater than $1$, the $q^{d/2-2}$-ovoids of $Q^-(d-1,q)$ arising from cases (1), (1') of Example \ref{exa_FR_M1compl} are projectively equivalent, cf. line 9 of Table \ref{tab_extfieldP'gt0}. Therefore, case (1') is essentially covered by case (1). We list them separately here to indicate that they arise from different field reductions.
\item[(2)] There are more cases where $\cP\setminus\cM_1$ is a transitive $m$-ovoid, cf. Examples \ref{exa_FR_HtoH_smallpara}, \ref{exa_FR_QMtoQM_smallpara},
\end{enumerate}
\end{remark}

We are now ready to state our classification of $m$-ovoids of finite classical polar spaces that admit a transitive automorphism group acting irreducibly on the ambient vector space. We do not consider the case where the ambient vector space $V$ has odd dimension $d$ and the field $\F_{q}$ has even characteristic, since we have the isomorphism $\textup{Sp}_{d-1}(q)\cong\Omega_d(q)$ which is consistent with their actions on the associated classical geometry.

\begin{thm}\label{thm_PGamUTransi}
	Suppose that $q=p^f$ with $p$ prime and $f$ even, $d>3$ and $(q,d)\ne(4,4)$. Let $\cM$ be an $m$-ovoid of $H(d-1,q)$ with a transitive automorphism group $H_0\le\PGamU_d(q^{1/2})$ such that its full preimage $H$ in $\Gamma(V):=\GamU_d(q^{1/2})$ acts irreducibly on the ambient space $V$. Then $d$ is odd, $E:=C_{\textup{End(V)}}(H^{(\infty)})\cong\F_{q^b}$ for a divisor $b>1$ of $d$, and we can regard $V$ as a vector space $V'$ over $E$. Moreover, there is a nondegenerate $H^{(\infty)}$-invariant Hermitian form on $V'$. If we write $\cP'=H(d/b-1,q^b)$ for the associated polar space, then $\cM$ is either in Example \ref{exa_Extension} or in Example \ref{exa_FR_HtoH_smallpara}.
\end{thm}

\begin{table}[!h]
\centering
\caption{Summary of parameters of transitive $m$-ovoids in $\cP=H(d-1,q)$, $(q,d)\neq(2^2,4)$}
\label{tab_sumHerm}	
\begin{tabular}{cccc}
\toprule
$m$                         & Condition                                     & Description           & Reference                                               \\ \hline
$\frac{q^{(d-b)/2}-1}{q-1}$ & $d$ odd,\, $d/b\ge 3$                          & $\cM_1$               & Example \ref{exa_Extension}           \\
$q^{(d-3)/2}$     & $q\in\{2^2,2^6\}$,\, $b=3$,\, $d\ge 9$ odd & $\cP\setminus\cM_{1}$ & Example \ref{exa_FR_HtoH_smallpara} \\ \bottomrule
\end{tabular}
\end{table}

\begin{thm}\label{thm_PGamOTransi}
	Suppose that $q=p^f$ with $p$ an odd prime and $d$ is odd with $d>3$. Let $\cM$ be an $m$-ovoid of $Q(d-1,q)$ with a transitive automorphism group $H_0\le\PGamO_d(q)$ such that its full preimage $H$ in $\Gamma(V):=\GamO_d(q)$ acts irreducibly on the ambient space $V$. Then it is one of the cases for $Q(d-1,q)$ in Examples \ref{exa_Imprimitive}, \ref{exa_Alternating}, \ref{exa_thmCross} and \ref{exa_thmNatural}.
\end{thm}
\begin{table}[!h]
\centering
\caption{Summary of parameters of transitive $m$-ovoids in $\cP=Q(d-1,q)$, $d$ odd}
\label{tab_sumQ}
\begin{tabular}{cccc}
\toprule
$m$      & $(q,d)$      & Description   & Reference \\ \hline
$2$      & $(3,5)$      & imprimitive   & Example \ref{exa_Imprimitive}  \\ \hline
$1,\cdots,6,8,9,10,12$ & $(3,7)$                 & \begin{tabular}[c]{@{}c@{}}imprimitive,\\ nearly simple\end{tabular} & Examples \ref{exa_Imprimitive}, \ref{exa_Alternating}, \ref{exa_thmCross}  \\ \hline
$6,15,16$   & $(5,7)$      & nearly simple & Example \ref{exa_thmCross}       \\ 
$1,q,q^2$& $(3^f,7)$,\, $f\ge 3$ odd & nearly simple &Example \ref{exa_thmNatural} \\
$1,q+q^2$& $(3^f,7)$    & nearly simple &   Example \ref{exa_thmNatural} \\ \bottomrule
\end{tabular}
\end{table}

\begin{thm}\label{thm_PGamOMinusTransi}
Suppose that $q=p^f$ with $p$ prime, and let $d$ be an even integer with $d\ge 6$ and $(q,d)\neq(2,6)$. Let $\cM$ be an $m$-ovoid of $\cP=Q^{-}(d-1,q)$ with a transitive automorphism group $H_0\le\PGamO^{-}_d(q)$ such that its full preimage $H$ in $\Gamma(V):=\GamO^{-}_d(q)$ acts irreducibly on the ambient space $V$. Then we have either of the following cases:
\begin{enumerate}
\item[(a)] $\cM$ is  one of the  $\PSL_{3}(4)$-invariant $2$-ovoids of $Q^-(5,3)$ in Example \ref{exa_thmCross},
\item[(b)] $E:=C_{\textup{End(V)}}(H^{(\infty)})\cong\F_{q^b}$ for a proper divisor $b$ of $d$.
\end{enumerate}
In the case (b), $V$ is a vector space $V'$ over $E$ equipped with a nondegenerate $H^{(\infty)}$-invariant form $\kappa'$. Let $\cP'$ be the associated polar space, and define $\cM_1$, $\cM_{2}$ as in \eqref{eqn_M1def} and \eqref{eqn_sec4cM2def} respectively. Then either $\cM$ is in Examples \ref{exa_Extension}, \ref{exa_FR_M1compl} and \ref{exa_FR_QMtoQM_smallpara}, or it is a $3$-ovoid of $Q^-(5,5)$ with $H^{(\infty)}=3^\cdot A_7$ and $\cP'=H(2,5^2)$.
\end{thm}

\begin{table}[!h]
\centering
\caption{Summary of parameters of transitive $m$-ovoids in $\cP=Q^{-}(d-1,q)$, $(q,d)\neq(2,6)$}
\label{tab_sumQM}
\begin{tabular}{cccc}
\toprule
$m$      & Conditions     & Description   & Reference \\ \hline
$2$	&   $(q,d)=(3,6)$     & nearly simple & Example \ref{exa_thmCross}\\
$3$ &$b=2$,\, $\cP'=H(2,5^2)$& field reduction & Theorem \ref{thm_PGamOMinusTransi}(b)\\
$\frac{q^{d/2-b}-1}{q-1}$& $d\geq6$, $1<d/b<d$ &$\cM_1$& Example \ref{exa_Extension}\\
$q^{d/2-2}$ & $b=2$, $d/2\geq 4$ even & $\cP\setminus\cM_1$ & Example \ref{exa_FR_M1compl}\\
$bfq^{d/2-b}$ & $(p^f,b)\in\{(2,3),(2^3,3)\}$ &$\cP\setminus\cM_{1}$&Example \ref{exa_FR_QMtoQM_smallpara}\\
$2^{d/2-2}$& $(q,b)=(2,4)$, $d/4 \ge 4$  &$\cP\setminus \cM_{2}$& Example \ref{exa_FR_QMtoQM_smallpara}\\\bottomrule	
\end{tabular}
\end{table}

\begin{thm}\label{thm_PGamOPlusTransi}
Suppose that $q=p^f$ with $p$ prime  and $d$ is even with $d\ge 6$. Let $\cM$ be an $m$-ovoid of $Q^{+}(d-1,q)$ with a transitive automorphism group $H_0\le\PGamO^{+}_d(q)$ whose full preimage $H$ in $\Gamma(V)=\GamO^{+}_d(q)$ acts irreducibly on the ambient space $V$. Then it is either in Examples \ref{exa_Imprimitive}, \ref{exa_Alternating}, \ref{exa_thmCross}, \ref{exa_thmNatural} and \ref{exa_Q+72}, or  $E:=C_{\textup{End(V)}}(H^{(\infty)})\cong\F_{q^2}$ with $d/2$ odd. In the last case, $V$ is a vector space $V'$ over $E$ equipped with a nondegenerate $H^{(\infty)}$-invariant quadratic form $\kappa'$ on $V'$. Let $\cP'$ be the associated polar space, and define $\cM_1$ as in \eqref{eqn_M1def}. Then it is in  either of Examples \ref{exa_Extension}, \ref{exa_FR_PSU3_QPlus7} and \ref{exa_FR_M1compl}.
\end{thm}

\begin{table}[!h]
\centering
\caption{Summary of parameters of transitive $m$-ovoids in $\cP=Q^{+}(d-1,q)$, $d\geq 6$}
\label{tab_sumQM}
\begin{tabular}{c|c|cc}
\toprule
$m$      & Conditions     & Description   & Reference \\ \hline
 $8,16,24,32$& $(q,d)=(3,8)$ & imprimitive &  Example \ref{exa_Imprimitive}\\\hline
$7,21,28,99$ & $(q,d)=(2,14)$ & nearly simple & Example \ref{exa_Alternating}\\\hline
$1,4,5,6,10,20,30,36$&\multirow{2}*{$(q,d)=(5,8)$} & \multirow{2}*{nearly simple} & \multirow{2}*{Examples \ref{exa_Alternating},\ref{exa_thmCross}}\\
$40,45,60,96,100,120$ &&& \\\hline
$1,q+q^2+q^3$ &$p=2$,\, $d=8$& nearly simple &  Example \ref{exa_thmNatural}\\\hline
$1,q+q^2,q^3$	& $q\equiv2\pmod3$,\, $d=8$ & nearly simple &  Example \ref{exa_thmNatural}\\\hline
	$\frac{q^{d/2-1}-1}{q-1}$	& $b=2$,\, $d/b\geq3$ odd & $\cM_1$ &  Example \ref{exa_Extension}\\\cline{1-1}\cline{3-4}
	$q^{d/2-1}$&  $\cP'=Q(d/2-1,q^2)$ &$\cP\setminus\cM_{1}$  &Example \ref{exa_FR_M1compl}\\\hline
	$q+1, q^6,q^2(q^2+1)(q+1)$	&$p=3$,\, $d=14$& field reduction&Example \ref{exa_FR_PSU3_QPlus7}\\\hline
	$1,2,3,4,6,8,9,12,14$& $(q,d)=(2,8)$ & - & Example \ref{exa_Q+72}\\\bottomrule
\end{tabular}
\end{table}

\begin{thm}\label{thm_PGamSpTransi}
	Suppose that $q=p^f$ with $p$ prime and let $d$ be an even integer with $d\ge 4$ and $(q,d)\neq(2,6)$. Let $\cM$ be an $m$-ovoid of $\cP=W(d-1,q)$ with a transitive automorphism group $H_{0}\leq\PGamSp_{d}(q)$ whose full preimage $H$ in $\Gamma(V):=\GamSp_{d}(q)$ acts irreducibly on ambient space $V$. Then there are three cases:
\begin{enumerate}
\item[(a)] $\cM$ is in either of Examples \ref{exa_Classical}, \ref{exa_Symplectic}, \ref{exa_thmsporadic}, \ref{exa_thmCross}, \ref{exa_thmNatural};
\item[(b)] $H\le \Gamma\textup{U}_1(q^{d/2})$ with $d=4$, or  $(p,d,f,m)=(2,6,8,24),(2,6,9,27)$;
\item[(c)]$E:=C_{\textup{End(V)}}(H^{(\infty)})\cong\F_{q^b}$ for a proper divisor $b$ of $d$.
\end{enumerate}
In the  case (c), $V$ is a vector space $V'$ over $E$ equipped with a nondegenerate $H^{(\infty)}$-invariant form $\kappa'$. Let $\cP'$ be the associated polar space, and define $\cM_1$, $\cM_{2}$ as in \eqref{eqn_M1def} and \eqref{eqn_sec4cM2def} respectively. Then $\cM$ is either in one of  Examples \ref{exa_Extension}, \ref{exa_FR_W_Sz}, \ref{exa_FR_M1compl}, \ref{exa_FR_HtoW_smallpara} and \ref{exa_FR_QMtoW_smallpara}, Theorems \ref{thm_II_HtoWcons} (2) and \ref{thm_II_Wqevenb2cons},
or it is one of the following cases:
	\begin{enumerate}
		\item[(c1)] $\cP'=H(2,5^2)$, $H^{(\infty)}=3^\cdot A_7$, $m=6$;
		\item[(c2)] $\cP'=H(2,3^2)$, $H^{(\infty)}=2^{\cdot}\textup{PSL}_2(7)$, $m=4$.
	\end{enumerate}
\end{thm}

\begin{table}[!h]
\centering
\caption{Summary of parameters of transitive $m$-ovoids in $\cP=W(d-1,q)$, $(q,d)\neq(2,6)$}
\label{tab_sumW}
\begin{tabular}{c|c|c|c}
\toprule
$m$      & Conditions     & Description   & Reference \\ \hline
$\frac{q^{d/2-1}-1}{q-1}$  &\multirow{2}*{$q$ even} & $Q^{-}(d-1,q)$ &\multirow{2}*{Example \ref{exa_Classical}}\\\cline{1-1} \cline{3-3}
$q^{d/2-1}$ & &$\cP\setminus Q^{-}(d-1,q)$ & \\\hline
	$2$	& $(q,d)=(3,4)$ &symplectic type &  Example \ref{exa_Symplectic}\\\hline
	$15,16$	& $(q,d)=(5,6)$ & nearly simple & Example \ref{exa_thmsporadic}\\\hline
	$6,9$ & $(q,d)=(2,8)$ & nearly simple  & Example \ref{exa_thmCross}\\\hline
	$12$  & $(q,d)=(2,12)$  & nearly simple  & Example \ref{exa_thmCross}\\\hline
	$1,q$	& $q=2^f$,$f\geq3$ odd & nearly simple & Example \ref{exa_thmNatural}\\\hline
\multirow{2}{*}{$\frac{q^{(d-b)/2}-1}{q-1}$}	& $d/b\geq3$ odd,\, $b>1$,  & \multirow{2}{*}{$\cM_1$} & \multirow{2}{*}{Example \ref{exa_Extension}}\\
& $\cP'=H(d/b-1,q^b)$  &  &\\\hline
\multirow{2}{*}{$q^{d/2-1}$}&$b=2$,\, $d/b\geq3$ odd, &\multirow{2}{*}{$\cP\setminus\cM_{1}$}& \multirow{2}{*}{Example \ref{exa_FR_M1compl}}\\
&  $\cP'=H(d/2-1,q^2)$&  &\\\hline
$3fq^{d/2-3}$&$p^f\in\{2,2^3\}$  &$Q^{-}(d-1,q)\setminus\cM_{1}$ &Examples \ref{exa_FR_HtoW_smallpara},\ref{exa_FR_QMtoW_smallpara}\\\hline
$2^{d/2-2}$	&$q=2$  &$Q^{-}(d-1,2)\setminus \cM_{2}$& Examples \ref{exa_FR_HtoW_smallpara}, \ref{exa_FR_QMtoW_smallpara}\\\hline
		$\frac{|\mathcal{T}|}{q-1}q^{(d-b)/2}$	 & $b=4$,\, $q$ even &  $\cM=\{\la v\ra_{\F_q}:\,H'(v)\in\mathcal{T}\}$  & Theorem \ref{thm_II_HtoWcons} (2)\\\hline	
	$\frac{|\mathcal{T}|}{q-1}q^{d/2-2}$ & $p=2$, $b=2$ &  $\cM=\{\la v\ra_{\F_q}:\,Q'(v)\in\mathcal{T}\}$ & Theorem \ref{thm_II_Wqevenb2cons}\\\hline
				
\multirow{2}{*}{$q^b\frac{q^b-1}{q-1}$,$\frac{q^b-1}{q-1}$}	&$q=2^f$,\, $f$ odd,&\multirow{2}{*}{ field reduction}&\multirow{2}{*}{Example \ref{exa_FR_W_Sz}}\\
& $d=4b$, $b>1$& &\\\hline	
$6$&$\cP'=H(2,5^2)$&field reduction & Theorem \ref{thm_PGamSpTransi} (c1)\\\hline
$4$&$\cP'=H(2,3^2)$& field reduction &  Theorem \ref{thm_PGamSpTransi} (c2)\\\hline
 \multicolumn{2}{l|}{$d\geq 4$ or $(p,d,f,m)=(2,6,8,24),(2,6,9,27)$} & - & Theorem \ref{thm_PGamSpTransi} (b)\\\bottomrule
\end{tabular}
\end{table}

We summarize the parameters of the transitive $m$-ovoids in Theorems \ref{thm_PGamUTransi}-\ref{thm_PGamSpTransi} in Tables \ref{tab_sumHerm}--\ref{tab_sumW}. It is noteworthy to mention that we obtain some unexpected new infinite families of $m$-ovoids of $W(d-1,q)$ in Theorems \ref{thm_II_HtoWcons} and \ref{thm_II_Wqevenb2cons}. We conclude this section with some comments on the above theorems.

\begin{remark}\label{rem_end_sec4}
Let $q=p^f$ and $d,e$ be as in Table \ref{tab_valuede}.
\begin{enumerate}
\item[(1)] We have excluded the cases with $(p,ef)\ne(2,6)$ from consideration in the above theorems, i.e., $H(3,4)$, $W(5,2)$, $Q^-(5,2)$, $Q^+(7,2)$. In $H(3,4)$, there are exactly two transitive ovoids, two transitive $2$-ovoids and one transitive $3$-ovoid up to projective equivalence, cf. \cite{Brouwer1990,DeBruyn2020}. By  \cite{Ceria2022}, the proper $m$-ovoids of $W(5,2)$ have parameter $m\in\{3,4\}$ and the transitive ones are exactly $Q^-(5,2)$ and its complement, cf. Example \ref{exa_Classical}. The polar space $Q^-(5, 2)$ thus has no proper $m$-ovoids, cf. Remark \ref{rem_2Fact}. The parameters of proper transitive $m$-ovoids in $Q^+(7,2)$ are listed in Example \ref{exa_Q+72}.
\item[(2)] We do not attempt to classify all transitive $m$-ovoids of $W(3,q)$ whose stabilizers lie in $\GamU^\#_1(q^2)$ here, cf. Theorem  \ref{thm_PGamSpTransi} (b). Please refer to \cite{cossidente2008m} for constructions of such transitive $2$-ovoids in $W(3,q)$. Also, we do not consider the transitive $m$-ovoids of $Q^+(3,q)$ due to lack of geometric significance.
\item[(3)] There are two cases with small parameters left open in $W(5,q)$.  They have automorphism groups lying in $\Gamma\textup{U}_1(q^3)$, cf. Theorem \ref{thm_PGamSpTransi} (b).
\end{enumerate}
\end{remark}

\section{Improved bounds in the extension field case}\label{sec:boundsection}
Let $\cP$ be one of the polar spaces: $H(d-1,q)$ with $d$ odd, $Q^-(d-1,q)$ and $W(d-1,q)$. In this section, we study whether certain subsets $\cM$ of $\cP$  are $m$-ovoids or not by exponential sum calculations. We shall derive improved bounds on the parameters by Kloosterman sum estimations, which will be needed for our discussions in the next two sections. In Theorem \ref{thm_II_HtoWcons}, we will obtain new families of $m$-ovoids in $W(d-1,q)$.\medskip

We first introduce some notation that will be used throughout this section. For a property $\chi$, we write $[[\chi]]:=1$ or $0$ according as the property $\chi$ holds or not.  Let $b$ be a divisor of $d$. We write $\tr_{q^b/q}(x):=x+x^q+\cdots+x^{q^{b-1}}$ for the trace function from $\F_{q^b}$ to $\F_q$. For each subfield $\F$ of $\F_{q^b}$, we write $\tr_{\F}$ for the absolute trace function of $\F$, and define $\psi_\F(x)=\exp(\frac{2\pi \sqrt{-1}}{p}\tr_{\F}(x))$ for $x\in\F$. In particular, we write $\psi:=\psi_{\F_q}$ as the canonical additive character of $\F_q$.

Let $V'$ be a $d/b$-dimensional vector space over $\F_{q^b}$ with a nondegenerate sesquilinear bilinear form or quadratic form $\kappa'$. Let $\cP'$ be the associated polar space. We regard $V'$ as a $d$-dimensional vector space $V$ over $\F_q$, and define a form $\kappa$ as in Table \ref{tab_extfieldP'gt0}. In the case $\kappa$ is a quadratic form, set $\f(x,y):=\kappa(x+y)-\kappa(x)-\kappa(y)$, and in the other cases set $\f=\kappa$. For each nonzero vector $u\in V$, the map $u\mapsto u^\perp$ is defined by $u^\perp=\{ v\in V:\,\f(u,v)=0\}$.
For each $a\in V$, define $\psi_a(x):=\psi(\f(a,x))$. Since $\f$ is nondegenerate, each additive character of $(V,+)$ is of the form $\psi_a$ for some $a\in V$. Let $\cM$ be a subset of $\cP$, and define
\begin{equation}\label{eqn_def_D}
	D=\{\lambda v:\,\lambda\in\F_q^*,\la v\ra\in\cM\}.
\end{equation}
For a nonzero vector $a\in V$, we have
\begin{align}\label{eqn_EFCHHpsiDM}
	\psi_{a}(D)=\sum_{\left\langle v\right\rangle \in\cM}\sum_{\lambda\in\F_{q}^{*}}\psi_{a}(\lambda v)=q\cdot |a^{\perp}\cap\cM|-|\cM|.
\end{align}
That is, $\psi_a(D)$ reflects the intersection property of $\cM$ with the hyperplane $a^\perp$.

Suppose that $\cP$ is one of the polar spaces: $H(d-1,q)$ with $d$ odd, $Q^-(d-1,q)$ or $W(d-1,q)$. By definition, $\cM$ is an $m$-ovoid of $\cP$ if and only if for each  $\left\langle a\right\rangle_{\F_{q}}\in\cP$, we have
\begin{equation}\label{eqn_SNEFCHHpsiDM}
	\psi_{a}(D)=\left\{\begin{array}{ll}-q^{d/2}+m(q-1), & \text{if }a\in D,\\m(q-1), &  \text{if } a\notin D.\end{array}\right.
\end{equation}
We shall need the following results on the parameter $m$.
\medspace
\begin{thm}[\cite{BambergTightsets2007},Theorem 13]
\label{thm_Bboundm}
	Let $\cP$ be one of the polar spaces $H(2r,q)$, $Q^-(2r+1,q)$ and $W(2r-1,q)$,  and let $\mathcal{O}$ be an $m$-ovoid of $\cP$. Then $m\ge \ell$, where $\ell$ is as given in the table below.
	\begin{table}[!h]
		\centering

		\begin{tabular}{cc}
			\toprule
			$\cP$ & $\ell$\\	
			\hline
			$H(2r,q)$     & $(-3+\sqrt{9+4q^{r+1/2}})/(2q-2)$ \\
			$Q^-(2r+1,q)$ & $(-3+\sqrt{9+4q^{r+1}})/(2q-2)$   \\
			$W(2r-1,q)$   & $(-3+\sqrt{9+4q^{r}})/(2q-2)$     \\ \bottomrule
		\end{tabular}
	\end{table}
\end{thm}

\begin{thm}[\cite{Gavrilyuk2021}, Theorem 1.1]\label{thm_modulareqn}
If $\cQ^-(2r+1,q)$ possesses an $m$-ovoid, then $F(m)\equiv 0 \pmod{q+1}$, where \[F(m)=\begin{cases} m^2-m,\quad &\text{if $r$ is odd};\\
		m^2,\quad &\text{if $r$ is even and $q$ is even};\\
		m^2+\frac{q+1}{2}m,\quad &\text{if $r$ is even and $q$ is odd}.
		\end{cases}\]
\end{thm}

We shall also need some results on Kloosterman sums, cf. Chapter 5 of \cite{LidlFF}. Let $\F$ be a subfield of $\F_{q^b}$, and $\psi_\F$ be its canonical additive character. For $a_{1},\,a_{2}\in\F$, the Kloosterman sum $K(\psi_\F,a_{1},a_{2})$ is defined as
\[
K(\psi_\F,a_{1},a_{2})=\sum_{x\in\F^*}\psi_\F(a_{1}x+a_{2}x^{-1}).
\]
If $a_{1}=a_{2}=0$, it equals $|\F|-1$; if exactly one of $a_{1},a_{2}$ is $0$, it equals $-1$. In general, if $a_{1}\ne 0$, then by \cite[Theorem 5.45]{LidlFF} we have
\begin{equation}\label{eqn_Kloos}
	|K(\psi_\F,a_{1},a_{2})|\leq 2\cdot|\F|^{1/2}.
\end{equation}
\vspace{3pt}
\begin{lemma}\label{lem_expsum_Herm'}
	Suppose that $q^b$ is a square, and set $\F=\F_{q^{b/2}}$. Take $\lambda\in\F^*$ and   $a\in\F_{q^b}$. Then we have
	\begin{align*}
		\sum_{x\in\F_{q^b}}\psi_{\F}(\lambda x^{q^{b/2}+1})\psi_{\F_{q^b}}(a^{q^{b/2}}x)
		=-q^{b/2}\psi_\F\left(-a^{1+q^{b/2}}\lambda^{-1}\right).
	\end{align*}
\end{lemma}
\begin{proof}
	First consider the case $q$ is odd. We observe that $\psi_{\F}(y)=\psi_{\F_{q^b}}(2^{-1}y)$ for $y\in\F$, so that the claim follows from \cite[Theorem 1 (ii)]{weilCoulter} in this case.
	
	Next consider the case where $q$ is even. Write $\tr:=\tr_{q^b/q^{b/2}}$. Take $\delta \in\F_{q^b}^*$ such that $\delta^{q^{b/2}}+\delta=1$. Then $\{1,\delta\}$ is a basis of $\F_{q^b}$ over $\F$.  Write $\Delta:=\sum_{x\in\F_{q^b}}\psi_{\F}(\lambda x^{q^{b/2}+1})\psi_{\F_{q^b}}(a^{q^{b/2}}x)$.  Since $\psi_\F(z^2)=\psi_\F(z)$ for $z\in\F$, we have
	\begin{align*}
		\Delta&=\sum_{u_1,u_2\in\F}\psi_{\F}\left(\lambda (u_1+u_2\delta)(u_1+u_2\delta^{q^{b/2}})+\tr(a^{q^{b/2}})u_1+\tr(a^{q^{b/2}}\delta)u_2\right)\\
		&=\sum_{u_1\in\F}\psi_\F(\lambda u_1^2+\tr(a)u_1)\sum_{u_2\in\F}\psi_\F(\lambda u_1u_2+u_2^2\lambda\delta^{1+q^{b/2}}+\tr(a\delta^{q^{b/2}})u_2)\\
		&=\sum_{u_1\in\F}\psi_\F(\lambda u_1^2+\tr(a^2)u_1^2)\sum_{u_2\in\F}\psi_\F\left(u_2^2\cdot(\lambda^2u_1^2+\lambda\delta^{1+q^{b/2}}+\tr(a^2\delta^{2q^{b/2}}))\right).
	\end{align*}
	The inner sum is $q^{b/2}$ if $\lambda^2u_1^2+\lambda\delta^{1+q^{b/2}}+\tr(a\delta^{q^{b/2}})^2=0$ and is $0$ otherwise. The desired expression of $\Delta$ follows by some tedious but routine calculations which we omit here. This completes the proof.
\end{proof}

Take $V'=\F_{q^b}^{d/b}$, and define the Hermitian form
\begin{equation}\label{eqn_ExtFieldHerForm}
	 \kappa'(x,y)=x_{1}y_{1}^{q^{b/2}}+x_{2}y_{2}^{q^{b/2}}+\ldots+x_{d/b}y_{d/b}^{q^{b/2}},\quad x,y\in V'.
\end{equation}
As conventional, we write $H'(x)=\kappa'(x,x)$.
\begin{lemma}\label{lem_psiD_Herm'}
	Take $V'$, $\kappa'$ and $H'$ as above. If $D=\{x\in V:\,H'(x)\in\mathcal{T}\}$ for a subset $\mathcal{T}$ of $\F_{q^{b/2}}^*$ and $\psi_a(x)=\psi_{\F_{q^b}}(\kappa'(a,x))$, then
	\begin{align}\label{eqn_psiD_Herm'}
		 \psi_{a}(D)=(-1)^{d/b}q^{(d-b)/2}\sum_{t\in\mathcal{T}}K(\psi_{\F_{q^{b/2}}},t,H'(a)),\, a \in V'\setminus\{\bf 0\}.
	\end{align}
\end{lemma}
\begin{proof}
	In this proof, we set $\F:=\F_{q^{b/2}}$ for brevity. Take a nonzero vector $a$. We have
	\begin{align}\label{eqn_EFCHHpsiD}
		\psi_{a}(D)&=\sum_{t\in\mathcal{T}}\sum_{x\in V'}[[H'(x)=t]]\psi_{a}(x)
		=\sum_{t\in\mathcal{T}}\sum_{x\in V'}q^{-b/2}\sum_{\lambda\in\F}\psi\left( \lambda(H'(x)-t)\right)\psi_{a}(x)\notag \\
		&=q^{-b/2}\sum_{t\in\mathcal{T}}\sum_{x\in V'}\psi_{a}(x)		 +q^{-b/2}\sum_{t\in\mathcal{T}}\sum_{\lambda\in\F^{*}}\psi_{\F}(-t\lambda)\sum_{x\in V'}\psi(\lambda H'(x))\psi_{a}(x).\notag
	\end{align}
	We have $\sum_{x\in V'}\psi_{a}(x)=0$ since $a\ne 0$. By Lemma \ref{lem_expsum_Herm'}, we have
	\begin{align*}
		\sum_{x\in V'}\psi_\F(\lambda H'(x))\psi_{a}(x)=&\prod_{i=1}^{d/b} \left(\sum_{x_i\in\F_{q^b}}\psi_{\F}(\lambda x_i^{q^{b/2}+1})\psi_{\F_{q^b}}(a_i^{q^{b/2}}x_i)\right)\\
		=&(-q^{b/2})^{d/b}\psi_\F(-H'(a) \lambda^{-1}).
	\end{align*}
	Therefore, we have
	\begin{align*}
		\psi_{a}(D)&=0+(-1)^{d/b}q^{(d-b)/2}\sum_{t\in\mathcal{T}}\sum_{\lambda\in \F^*}\psi_{\F}(-t\lambda-H'(a)\lambda^{-1})\\
		&=(-1)^{d/b}q^{(d-b)/2}\sum_{t\in\mathcal{T}} K(\psi_\F,t,H'(a)).
	\end{align*}
	This completes the proof.
\end{proof}

\begin{lemma}\label{lem_Herm'_Na}
	Take $V'$, $\kappa'$  and $H'$ as above, and assume that $d/b$ is odd. For $a\in\F_{q^{b/2}}$, let $N(a)$ be the size of $\{x\in V':\,H'(x)=a\}$. Then $N(a)=q^{(d-b)/2}(q^{d/2}+1)$ for $a\ne 0$.
\end{lemma}
\begin{proof}
	Since $H'(\lambda x)=\lambda^{q^{b/2}+1}H'(x)$ for $\lambda\in\F_{q^b}^*$, we see that $N(a)=N(1)$ for all $a\in\F_{q^{b/2}}^*$.  Therefore, $(q^{b/2}-1)\cdot N(1)+N(0)=q^d$. On the other hand, $N(0)=1+(q^b-1)\cdot|H(d/b-1,q^b)|$, and $|H(d/b-1,q^b)|$ is available in Table \ref{tab_extfieldP'gt0}. The claim then follows.
\end{proof}

\subsection{The case $\cP=H(d-1,q)$, $\cP'=H(d/b-1,q^b)$, $d$ odd}\label{subsec_Hdodd}
We take $V'=\F_{q^b}^{d/b}$, let $\kappa'$ be as defined in \eqref{eqn_ExtFieldHerForm}, and set $\kappa=\tr_{q^b/q}\circ \kappa'$. As conventional, we write $H'(x):=\kappa'(x,x)$.
Let $\mathcal{T}$ be an $\F_{q^{1/2}}^{*}$-invariant subset of $\F_{q^{b/2}}^*$ such that $\tr_{q^{b/2}/q^{1/2}}(x)=0$ for each $x\in \mathcal{T}$. Take a set $T$ of $\F_{q^{1/2}}^*$-coset representatives of $\mathcal{T}$, so that $\mathcal{T}=\F_{q^{1/2}}^*\cdot T$. We set $\cM=\{\la v\ra_{\F_q}:\,H'(v)\in\mathcal{T}\}$
and correspondingly $D=\{x\in V:\,H'(x)\in\mathcal{T}\}$. Since $\frac{|\mathcal{T}|}{q-1}=\frac{|T|}{q^{1/2}+1}$, Lemma \ref{lem_Herm'_Na} yields
\begin{lemma}\label{lem_Herm_Msize}
	We have $|\cM|=\frac{|T|}{q^{1/2}+1}q^{(d-b)/2}(q^{d/2}+1)$.
\end{lemma}

\begin{proposition}\label{prop_Herm_Tbound}
	Take notation as above. If $\cM$ is a $\frac{|T|}{q^{1/2}+1}q^{(d-b)/2}$-ovoid of $H(d-1,q)$, then $|T|\ge\frac{q^{b/2}}{(q^{1/2}-1)(2q^{b/4}+1)}$.
\end{proposition}
\begin{proof}
	In this proof, we set $\F:=\F_{q^{b/2}}$ for brevity. We have $m=\frac{|T|}{q^{1/2}+1}q^{(d-b)/2}$ by Lemma \ref{lem_Herm_Msize}. Take $a\in D$, so that $\psi_a(D)=-q^{d/2}+m(q-1)$ by \eqref{eqn_SNEFCHHpsiDM}. On the other hand, we have $\psi_{a}(D)=-q^{(d-b)/2}\sum_{t\in\mathcal{T}}K(\psi_\F,t,H'(a))$ by Lemma \ref{lem_psiD_Herm'}. By the inequality in \eqref{eqn_Kloos}, we have
	\[
	|\psi_a(D)|\le q^{(d-b)/2}\cdot 2q^{b/4}(q^{1/2}-1)|T|=2q^{d/2-b/4}(q^{1/2}-1)|T|.
	\]
	The desired bound on $|T|$ follows after simplification. This completes the proof.
\end{proof}

By \eqref{eqn_SNEFCHHpsiDM} and \eqref{eqn_psiD_Herm'}, $\cM$ is a $\frac{|T|}{q^{1/2}+1}q^{(d-b)/2}$-ovoid if and only if
\begin{equation}\label{eqn_HtoH_psiDcond}
	\sum_{t\in\mathcal{T}} K(\psi_{\F_{q^{b/2}}},t,y)=\begin{cases}q^{b/2}-|\mathcal{T}|, & \text{if } y\in \mathcal{T},\\ -|\mathcal{T}|, &  \text{if } \tr_{q^{b/2}/q^{1/2}}(y)=0,\, y\not\in \mathcal{T}.\end{cases}
\end{equation}
The condition holds trivially when $y=0$. We observe that this condition is independent of $d/b$, so that we can run an exhaustive search for a given $(p,b,f,|T|)$ tuple.

\begin{example}\label{exa_FR_HtoH_smallpara}
	If $(p,b,f,|T|)$ is one of $(2,3,2,3),\,(2,3,6,9)$, then there is an element $\gamma$ of $\F_{q^{b/2}}$ such that $\mathcal{T}=\{\gamma^{2^i}z:\,0\le i\le \frac{bf}{2}-1,\,z\in\F_{q^{1/2}}^*\}$ has size $\frac{bf}{2}(q^{1/2}-1)$ and \eqref{eqn_HtoH_psiDcond} holds. Let $\cM_1$ be as in \eqref{eqn_M1def}. Since $\cM$ is a subset of $\cP\setminus \cM_1$, we deduce that $\cM=P\setminus\cM_1$ by comparing sizes in both cases.
\end{example}

\subsection{The case $\cP=Q^-(d-1,q)$, $\cP'=Q^-(d/b-1,q^b)$, $d/b\ge 4$}\label{subsec_QMdbeven}

Suppose that $V'=\F_{q^b}^{d/b}$, $\kappa'=Q'$, $\kappa=Q=\tr_{q^b/q}\circ Q'$, where
\[
Q'(x)=x_{1}x_{2}+x_{3}x_{4}+\ldots+x_{d/b-1}^2+a_{0}x_{d/b-1}x_{d/b}+a_{1}x_{d/b}^2,
\]
and $x_{d/b-1}^2+a_{0}x_{d/b-1}x_{d/b}+a_{1}x_{d/b}^2$ is irreducible over $\F_{q^b}$.
Let $\square$ be the set of nonzero squares of $\F_q$; in particular, $\square=\F_{q}^*$ if $q$ is even. Let $\mathcal{T}$ be a $\square$-invariant subset of $\F_{q^b}^*$ whose elements have relative trace $0$ to $\F_q$. Let $T$ be a set of $\square$-coset representatives of $\mathcal{T}$, so that $\mathcal{T}=\square\cdot T$. We set $\cM=\{\la v\ra_{\F_q}:\,Q'(v)\in\mathcal{T}\}$ and correspondingly $D=\{x\in V:\,Q'(x)\in\mathcal{T}\}$.

\begin{lemma}\label{lem_QMinus_Msize}
	$|\cM|=\epsilon|T|q^{d/2-b}(q^{d/2}+1)$, where $\epsilon=\frac{1}{2}$ or $1$ according as $q$ is odd or even.
\end{lemma}
\begin{proof}
	If $q$ is even, then $Q'(\lambda x)=\lambda^2Q'(x)$ for $\lambda\in\F_{q^b}^*$, so the number of solutions to $Q'(x)=a$ is a constant for any nonzero $a$. If $q$ is odd, then there is a similarity $g$ such that $Q'(x^g)=\mu Q'(x)$ for a nonsquare $\mu$ by p.39, (2.8.2) of \cite{KleidmanThemaximal8}, so the same conclusion holds. We then argue as in Lemma \ref{lem_Herm'_Na} and get the desired expression of $|\cM|$.
\end{proof}

\begin{proposition}\label{prop_Ellip_Tbound}
	Take notation as above, and set $\epsilon=\frac{1}{2}$ or $1$ according as $q$ is odd or even. If $\cM$ is an $\epsilon|T|q^{d/2-b}$-ovoid of $Q^{-}(d-1,q)$, then $|T|\ge\frac{q^b}{\epsilon(q-1)(2q^{b/2}+1)}$.
\end{proposition}
\begin{proof}
	We only give the proof for $q$ odd, and the case $q$ even is similar. In this proof, we write $E=\F_{q^b}$ for brevity. Take $\gamma$ to be a fixed nonsquare of $\F_q$ and we specify $a_0=0$, $a_1=-\gamma$ in the definition of $Q'$. Take $\la a\ra\in\cP$. We have
	\begin{align}
		\psi_{a}(D)&=\sum_{t\in\mathcal{T}}\sum_{x\in V'}[[Q'(x)=t]]\psi_{a}(x)\notag\\
		&=q^{-b}\sum_{t\in\mathcal{T}}\sum_{\lambda\in\F_{q^{b}}}
		\psi_{E}(-\lambda t)\sum_{x\in V'}\psi_{E}(\lambda Q'(x))\psi_{a}(x).
	\end{align}
	When $\lambda=0$, the summation of $\psi_{a}(x)$ over $x\in V'$ contributes $0$ because $a\ne0$. If $\lambda$ is nonzero, then
	\begin{align*}
		\sum_{x\in V'}\psi_{E}&(\lambda Q'(x))\psi_{a}(x)=\sum_{x\in V'}\psi_{E}\left(\lambda Q'(x)+B'(x,a) \right)\\
		&=\sum_{x\in V'}\psi_{E}\left(\lambda (Q'(x+a\lambda^{-1})-Q'(a\lambda^{-1}))\right)=\psi_{E}(-\lambda^{-1}Q'(a))\sum_{x\in V'}\psi_{E}\left(\lambda Q'(x)\right),
	\end{align*}
	where we replace $x+a\lambda^{-1}$ by $x$ in the second equality.  We compute that
	\begin{align*}
		\sum_{x\in V'}\psi_{E}\left(\lambda Q'(x)\right)=&
		\prod_{i=1}^{d/(2b)-1}\left(\sum_{x_{2i-1},x_{2i}}\psi_{E}(\lambda x_{2i-1}x_{2i})\right )
		\cdot\sum_{x_{d/b-1}}\psi_{E}(\lambda x_{d/b-1}^2)\sum_{x_{d/b}}\psi_{E}(-\lambda\gamma x_{d/b}^2)\\
		=&q^{d/2-b}\eta(-\gamma)G_E(\eta)^2=-q^{d/2}.
	\end{align*}
	Here, $\eta$ is the quadratic character of $E^*$ and $G_E(\eta)$ is the associated Gauss sum, cf. \cite{LidlFF}. Therefore, we have
	\begin{align}\label{eqn_EFCQQpsiDqodd}		 \psi_{a}(D)&=-q^{d/2-b}\sum_{t\in\mathcal{T}}\sum_{\lambda\in\F_{q^{b}}^*} \psi_{E}(-t\lambda-\lambda^{-1}Q'(a))=-q^{d/2-b}\sum_{t\in\mathcal{T}}K(\psi_{E},t,Q'(a)).
	\end{align}
	Now take $a\in D$, so that $\psi_a(D)=-q^{d/2}+m(q-1)$ by \eqref{eqn_EFCHHpsiDM}. We thus have $q^{d/2}-m(q-1)\le \frac{1}{2}(q-1)|T|q^{d/2-b}\cdot 2q^{b/2}$ by \eqref{eqn_Kloos}. We have $m=|T|q^{d/2-b}$ by Lemma \ref{lem_QMinus_Msize}, and the desired bound on $|T|$ then follows. This completes the proof.
\end{proof}
We remark that in the proof of Proposition \ref{prop_Ellip_Tbound}, the expression of $\psi_a(D)$ in the case $q$ even is exactly the same as  \eqref{eqn_EFCQQpsiDqodd}. This is due to the fact that $\sum_{x\in V'}\psi_{E}\left(\lambda Q'(x)\right)=-q^{d/2}$ also holds when $q$ is even. We leave the details to the interested reader. By \eqref{eqn_EFCHHpsiDM} and \eqref{eqn_EFCQQpsiDqodd}, $\cM$ is a $|T|q^{d/2-b}$-ovoid if and only if
\begin{equation}\label{eqn_QMtoQM_psiDcond}
	\sum_{t\in\mathcal{T}}K(\psi_{\F_{q^2}},t,y)
	=\begin{cases}q^b-|\mathcal{T}|,\,&\textup{if } y\in\mathcal{T},\\
		-|\mathcal{T}|,\,&\textup{if }\tr_{q^{b}/q}(y) = 0,\,y\not\in\mathcal{T}.\end{cases}
\end{equation}
The condition holds trivially when $y=0$.

\begin{remark}\label{rem_I_QMtoQM_b2}
	Take notation as above, and assume that $b=2$ and $\cM$ is an $m$-ovoid. The set of elements of relative trace $0$ to $\F_q$ is $\F_{q}\cdot\delta$, where $\delta\in\F_{q^2}^*$ with $\delta+\delta^q=0$. If $q$ is odd, then $\mathcal{T}$ is either $\F_q^*\cdot\delta$, $\square\cdot\delta$ or its complement. Since $m=\frac{|\mathcal{T}|}{q-1}q^{d/2-b}$ is an integer, we deduce that $\mathcal{T}=\F_q^*\cdot\delta$. Therefore, if $\cM$ is an $m$-ovoid of the prescribed form, then $\cM=\{\la v\ra:\,Q(v)=0,\,Q'(v)\ne 0\}$.
\end{remark}

\begin{example}\label{exa_FR_QMtoQM_smallpara}
	If $(p,b,f,|T|)$ is one of $(2,3,1,3)$, $(2,3,3,9)$, $(2,4,1,4)$, then we check with Magma that there is an element $\gamma$ of $\F_{q^{b}}$ such that $\mathcal{T}=\{\gamma^{2^i}z:\,0\le i\le bf-1,\,z\in\F_{q}^*\}$ has size $bf(q-1)$ and \eqref{eqn_QMtoQM_psiDcond} holds. In these three cases, $\cM=\{\la v\ra_{\F_q}:\,Q'(x)\in\mathcal{T}\}$ is a $bfq^{d/2-b}$-ovoid of $Q^-(d-1,q)$ if $d/b$ is even and $d/b\ge 4$. Let $\cM_1$ be as in \eqref{eqn_M1def}, i.e., $\cM_1=\{\la v\ra_{\F_q}:\,Q'(x)=0\}$. Since $\cM$ is a subset of $\cP\setminus \cM_1$, we deduce that $\cM=P\setminus\cM_1$ by comparing sizes in the first two cases. For the third case, we have $\cM=\cP\setminus\cM_2$, where
\begin{equation}\label{eqn_sec4cM2def}
\cM_2=\{\la v\ra_{\F_q}:\,\tr_{q^4/q^2}(Q'(x))=0\}.
\end{equation}
\end{example}

\subsection{The case $\cP=Q^-(d-1,q)$, $\cP'=H(d/b-1,q^b)$, $b$ even, $d/b$ odd}\label{subsec_QMHdbodd}
In this case, take $V'=\F_{q^b}^{d/b}$, $\kappa=\tr_{q^{b/2}/q}\circ H'$, where  $H'(x)=\kappa'(x,x)$ and  $\kappa'$ is as in  \eqref{eqn_ExtFieldHerForm}. The polar form $\f$ of $\kappa$ is $\f(x,y)=\tr_{q^b/q}(\kappa'(x,y))$. Moreover, we have
\[
H'(cx)=c^{1+q^{b/2}}H'(x)=c^2H'(x),\quad c\in\F_q.
\]
Let $\square$ be the set of nonzero squares of $\F_q$; in particular, $\square=\F_q^*$ if $q$ is even. Let $\mathcal{T}$ be a $\square$-invariant subset of $\F_{q^{b/2}}^*$ whose elements have relative trace $0$ to $\F_q$. Let $T$ be a set of $\square$-coset representatives $\mathcal{T}$, so that $\mathcal{T}=\square\cdot T$.  We set $\cM=\{\la v\ra_{\F_q}:\,H'(v)\in\mathcal{T}\}$, and correspondingly $D=\{v\in V:\,H'(v)\in\mathcal{T}\}$. By Lemma \ref{lem_Herm'_Na}, we deduce that $|\cM|=m(q^{d/2}+1)$ with $m=\frac{|\mathcal{T}|}{q-1}q^{(d-b)/2}$, i.e.,  $m=\epsilon|T|q^{(d-b)/2}$ with $\epsilon=\frac{1}{2}$ or $1$ according as $q$ is odd or even.

Take a nonzero vector $a$. We observe that
\[
\psi_a(x)=\psi_{\F_q}(\f(a,x))= \psi_{\F_{q^b}}(\kappa'(a,x)).
\]
By Lemma \ref{lem_psiD_Herm'}, $\psi_a(D)$ is the same as in \eqref{eqn_psiD_Herm'} with $\mathcal{T}$ as specified here. In the case $a\in D$, we obtain the following bound by \eqref{eqn_EFCHHpsiDM} and \eqref{eqn_Kloos}.
\begin{proposition}\label{prop_EllipHerm_Tbound}
	Take notation as above, and set $\epsilon=\frac{1}{2}$ or $1$ according as $q$ is odd or even. If $\cM$ is an $\epsilon|T|q^{(d-b)/2}$-ovoid of $Q^{-}(d-1,q)$, then $|T|\ge\frac{q^{b/2}}{\epsilon(q-1)(2q^{b/4}+1)}$.
\end{proposition}

By \eqref{eqn_EFCHHpsiDM}, $\cM$ is an $\epsilon|T|q^{(d-b)/2}$-ovoid if and only if
\begin{equation}\label{eqn_II_psiDcond_HtoQM}
	\sum_{t\in\mathcal{T}} K(\psi_{\F_{q^{b/2}}},t,u)
	=\begin{cases}q^{b/2}-|\mathcal{T}|,\,&\textup{ if }u\in \mathcal{T},\\
		-|\mathcal{T}|,\,&\textup{ if }\tr_{q^{b/2}/q}(u)=0,\,u\not\in\mathcal{T}.\end{cases}
\end{equation}
If $u=0$, then the condition holds trivially.
\begin{example}\label{exa_FR_HtoQM_smallpara}
	In the case $(p,b,f,|T|)$ is one of $\{(2,6,1,3),\,(2,6,3,9),(2,8,1,4)\}$, we check with Magma that there is an element $\gamma$ of $\F_{q^{b/2}}$ such that $\mathcal{T}=\{\gamma^{2^i}z:\,0\le i\le bf/2-1,\,z\in\F_{q}^*\}$ has size $\frac{1}{2}bf(q-1)$ and \eqref{eqn_II_psiDcond_HtoQM} holds. In these three cases, $\cM=\{\la v\ra_{\F_q}:\,H'(x)\in\mathcal{T}\}$ is a $\frac{1}{2}bfq^{(d-b)/2}$-ovoid of $Q^-(d-1,q)$ if $d/b$ is odd. They are projectively equivalent to those in Example \ref{exa_FR_QMtoQM_smallpara}.
\end{example}

\begin{remark}\label{rem_I_HtoQM_b24}
	For later use in Section \ref{subsec_II_QM}, we consider the case where $b\le 4$ and $\cM$ is an $m$-ovoid. If $b=2$, then no element of $\F_q^*$ has relative trace $0$ to $\F_q$, so this case is impossible.  Suppose that  $b=4$. Take $\delta\in\F_{q^2}^*$ such that $\delta+\delta^q=0$. We have that $\tr_{q^2/q}(y)=0$ if and only if $y\in\F_q\cdot\delta$. Therefore, $\mathcal{T}$ is either $\F_q^*\cdot\delta$,  $\square\cdot\delta$ or its complement. Since $m=\frac{|\mathcal{T}|}{q-1}q^{(d-b)/2}$ is an integer, we deduce that $\mathcal{T}=\F_q^*\cdot\delta$ if $q$ is odd. To sum up, we must have $\mathcal{T}=\F_q^*\cdot\delta$ and correspondingly $\cM=\{\la v\ra:\,Q(v)=0,\,H'(v)\ne 0\}$ when $b=4$ for $\cM$ to be an $m$-ovoid.
\end{remark}

\subsection{The case $\cP=W(d-1,q)$, $\cP'=H(d/b-1,q^b)$, $b$ even, $d/b$ odd}\label{subsec_HtoW_I}

In this case, $\kappa=\tr_{q^b/q}\circ\lambda \kappa'$ for some $\lambda\in\F_{q^b}^*$ such that $\lambda+\lambda^{q^{b/2}}=0$. Set $V=\F_{q^b}^{d/b}$, and let $\kappa'$ be as in \eqref{eqn_ExtFieldHerForm}. As conventional, we write $H'(x)=\kappa'(x,x)$. Since $b$ is even, we have
\[
H'(cx)=c^{1+q^{b/2}}H'(x)=c^2H'(x),\quad c\in\F_q.
\]
Let $\mathcal{T}$ be a $\square$-invariant subset of $\F_{q^{b/2}}^*$, where $\square$ is the set of nonzero squares of $\F_q$. Take a set $T$ of $\square$-coset representatives of $\mathcal{T}$, so that $\mathcal{T}=\square\cdot T$. We define
$\cM=\{\la v\ra_{\F_q}:\,H'(x)\in\mathcal{T}\}$ and correspondingly $D=\{x\in V:\,H'(x)\in\mathcal{T}\}$. Then  $|\cM|=m(q^{d/2}+1)$ with $m=\epsilon |T|q^{(d-b)/2}$ by Lemma \ref{lem_Herm'_Na}, where $\epsilon=1/2$ or $1$ according as $q$ is odd or even. For a nonzero vector $a$, we observe that
\[
\psi_a(x)=\psi_{\F_q}(\kappa(a,x))=\psi_{\F_{q^b}}(\lambda \kappa'(a,x))=\psi_{\F_{q^b}}( \kappa'(\lambda a,x)).
\]
Together with Lemma \ref{lem_psiD_Herm'}, we deduce that
\[
\psi_{a}(D)=-q^{(d-b)/2}\sum_{t\in\mathcal{T}} K(\psi_{\F_{q^{b/2}}},t,H'(\lambda a)).
\]
In the case $a\in D$,  we obtain the following bound
by \eqref{eqn_EFCHHpsiDM} and \eqref{eqn_Kloos}.
\begin{proposition}\label{prop_HtoW_Tbound}
	Take notation as above, and set $\epsilon=1/2$ or $1$ according as $q$ is odd or even. If $\cM$ is an $\epsilon|T|q^{(d-b)/2}$-ovoid of $W(d-1,q)$, then $|T|\ge\frac{q^{b/2}}{\epsilon(q-1)(2q^{b/4}+1)}$.
\end{proposition}

By \eqref{eqn_EFCHHpsiDM}, $\cM$ is an $\epsilon|T|q^{(d-b)/2}$-ovoid if and only if
\begin{equation}\label{eqn_II_psiDcond_HtoW}
	\sum_{t\in\mathcal{T}} K(\psi_{\F_{q^{b/2}}},t,-\lambda^2u)
	=\begin{cases}q^{b/2}-|\mathcal{T}|,\,\textup{ if }u\in \mathcal{T}\\
		-|\mathcal{T}|,\,\textup{ if }u\in\F_{q^{b/2}}\setminus\mathcal{T}\end{cases}.
\end{equation}
If $u=0$, then the condition holds trivially.
\begin{example}\label{exa_FR_HtoW_smallpara}
	In the case $(p,b,f,|T|)$ is one of $\{(2,6,1,3),\,(2,6,3,9),(2,8,1,4)\}$, we check with Magma that there is an element $\gamma$ of $\F_{q^{b/2}}$ such that $\mathcal{T}=\{\gamma^{2^i}z:\,0\le i\le \frac{bf}{2}-1,\,z\in\F_{q}^*\}$ has size $bf(q-1)$ and \eqref{eqn_II_psiDcond_HtoW} holds. In these three cases, $\cM=\{\la v\ra_{\F_q}:\,H'(v)\in\mathcal{T}\}$ is a $bfq^{(d-b)/2}$-ovoid of $W(d-1,q)$ if $d/b$ is odd. They are projectively equivalent to the ones arising from the $m$-ovoids of $Q^-(d-1,q)$ in Example \ref{exa_FR_QMtoQM_smallpara}, cf. Remark \ref{rem_2Fact} (2).
\end{example}

\begin{remark}\label{rem_I_HtoW_b2}
	Take notation as above, and suppose that $b=2$ and $\cM$ is an $m$-ovoid. In this case, we must have $\mathcal{T}\in\{\F_q^*,\square,\F_q^*\setminus\square\}$. Since $m=\frac{|\mathcal{T}|}{q-1}q^{(d-b)/2}$ is an integer, we deduce that $\mathcal{T}=\F_q^*$, i.e., $\cM=\{\la v\ra:\,H'(v)\ne 0\}$. It is indeed a transitive $q^{(d-b)/2}$-ovoid.
\end{remark}

\begin{lemma}\label{lem_II_HtoW_qodd}
Take notation as above, and assume that $b=4$ and $q$ is odd. Then the elements of $\mathcal{T}$ can not be all squares or nonsquares of $\F_{q^2}^*$.	
\end{lemma}
\begin{proof}
Take a complete set of $\square$-coset representation $R$ of $\F_{q^2}^*$, and set $T=R\cap \mathcal{T}$. We have $\F_{q^2}^*=\square\cdot R$, and $\mathcal{T}=\square\cdot T$. Since $\lambda+\lambda^{q^2}=0$, we deduce that $\lambda^2$ is a nonsquare of $\F_{q^2}$, and $\delta:=\lambda^{1+q}$ is in $\F_{q^2}$ with relative trace $0$ to $\F_q$. Write $E=\F_{q^2}$, $F=\F_q$, $\tr=\tr_{q^2/q}$, and let $\eta$ be the quadratic character of $\F_{q}^*$. We set $\eta(0):=0$. We observe that $\tr(x)=0$ if and only if $x\in\F_q\cdot\delta$. Set $\Delta_u:=\sum_{t\in\mathcal{T}} K(\psi_{E},t,-\lambda^2u)$. The condition \eqref{eqn_II_psiDcond_HtoW} holds trivially when $u=0$, so assume that $u\ne 0$.  We compute that
	\begin{align*}
		\Delta_u&=\frac{1}{4}\sum_{r\in R,t\in T}\sum_{y,z\in F^*}\psi_E(y^2tz^2r-\lambda^2ur^{-1}z^{-2})=\frac{1}{4}\sum_{r\in R,t\in T}\sum_{y\in F^*}\psi_E(y^2tr)\sum_{z\in F^*}\psi_E(-\lambda^2ur^{-1}z^{-2})\\
		&=\frac{1}{4}\sum_{r\in R,t\in T}\sum_{y\in F^*}\psi_F(y^2\tr(tr))\sum_{z\in F^*}\psi_F(z^{-2}\tr(-\lambda^2ur^{-1})).
	\end{align*}
	Here we replaced the variable $yz$ by $y$ in the second equality. We have $\sum_{y\in F^*}\psi_F(y^2\tr(tr))=-1+G_F(\eta)\eta(\tr(rt))+q[[\tr(rt)=0]]$, where $G_F(\eta)$ is the Gauss sum associated to $\eta$, cf. \cite{LidlFF}. The sum $\sum_{z\in F^*}\psi_F(z^{-2}\tr(-\lambda^2ur^{-1}))$ has a similar expression. After expansion and simplification, we deduce that
	\begin{align*}
		4\Delta_{u}=-2(q-1)|T|+q\cdot S_u+2q^2\cdot\#\{t\in T:\,t\in \F_q^*\cdot\lambda^{2}u^{-1}\}
	\end{align*}
	with $S_u=\sum_{r\in R,t\in T}\eta(\tr(rt))\cdot\eta(\tr(\lambda^2ur^{-1}))$. Here, most terms vanish due to the fact that there are exactly two elements $r\in R$ such that $\tr(\lambda^2ur^{-1})=0$ (or $\tr(rt)=0$) which differ by a nonsquare of $\F_q$. For instance, $\sum_{r\in R}\eta(\tr(rt))[[\tr(\lambda^2ur^{-1})=0]]=0$ for each $t\in T$. The value of $4\Delta_u$ should be equal to $4q^2-2(q-1)|T|$ or $-2(q-1)|T|$ according as $u\in\mathcal{T}$ or not by \eqref{eqn_II_psiDcond_HtoW}.

	First suppose that $\mathcal{T}$ is not $\F_q^*$-invariant. Then there is $u\in\mathcal{T}$ such that $\delta^2u\not\in\mathcal{T}$. From the expressions of $\Delta_u$ and $\Delta_{\delta^2u}$ and the fact $S_{\delta^2u}=\eta(\delta^2)S_u=-S_u$ and $\delta^2\in\F_{q}^{*}$, we deduce that $S_u+2q\cdot\#\{t\in T:\,t\in \F_q^*\lambda^{2}u^{-1}\}=4q$, $-S_u+2q\cdot\#\{t\in T:\,t\in \F_q^*\cdot\lambda^{2}u^{-1}\}=0$. It follows that $S_u=2q$, and there is exactly one $t\in T$ such that $t\in \F_q^*\cdot\lambda^{2}u^{-1}$. It follows that $ut$ is in $\F_q^*\cdot\lambda^2$ whose elements are nonsquares of $\F_{q^2}^*$. This proves the claim in this case.
	
	Next suppose that $\mathcal{T}$ is $\F_q^*$-invariant. In this case, $S_u=0$ by the fact that  both $\square\cdot t$ and $\square\cdot\delta^2t$ are contained in $\mathcal{T}$ and $\eta(\delta^2)=-1$. The condition for $\cM$ to be an $m$-ovoid then reduces to $\#\{t\in T:\,t\in \F_q^*\lambda^{2}u^{-1}\}=2$ or $0$ according as $u$ is in $\mathcal{T}$ or not. In other words, $\lambda^{2}u^{-1}\in\mathcal{T}$ if and only if $u\in\mathcal{T}$. The product of $\lambda^{2}u^{-1}$ and $u$ is a nonsquare of $\F_{q^2}^*$. This completes the proof.
\end{proof}

\begin{thm}\label{thm_II_HtoWcons}
Take notation as above and assume that $b=4$. Set $\lambda=1$ if $q$ is even.
	\begin{enumerate}
		\item[(1)] If $q$ is odd and $\mathcal{T}$ is $\F_q^*$-invariant, then $\cM$ is a $\frac{|\mathcal{T}|}{q-1}q^{d/2-2}$-ovoid if and only if   $\mathcal{T}=\{\lambda^2t^{-1}:\,t\in\mathcal{T}\}$;
		\item[(2)] If $q$ is even, then $\cM$ is a $\frac{|\mathcal{T}|}{q-1}q^{d/2-2}$-ovoid if and only if $\mathcal{T}=\{t^{-1}:\,t\in\mathcal{T}\}$.
	\end{enumerate}
\end{thm}
\begin{proof}
	(1) is established in the last paragraph of the proof of Lemma \ref{lem_II_HtoW_qodd}.
	We now prove (2). Suppose that $q$ is even. In this case, $\square=\F_q^*$. Take $T$ as a set of $\F_q^*$-coset representatives of $\mathcal{T}$ such that $\mathcal{T}=\F_q^*\cdot T$. Write $T=\{\lambda_1,\dots,\lambda_{|T|}\}$, $E=\F_{q^2}$ and $\tr=\tr_{q^2/q}$. Since $q$ is even, $\tr(y)=0$ if and only if $y\in\F_{q}$. We now calculate the left hand side of \eqref{eqn_II_psiDcond_HtoW}. It holds trivially if $u=0$, so we assume $u\ne 0$. We have
	\begin{align*}
		\sum_{t\in T}\sum_{y\in\F_{q}^{*}}K(\psi_{E},yt,u)&=\sum_{z\in E^{*}}\sum_{t\in T}\psi_{E}(uz^{-1})\sum_{y\in\F_{q}^{*}}\psi_{\F_{q}}(\tr(tz)y)\\&=\sum_{z\in E^{*}}\sum_{t\in T}\psi_{E}(uz^{-1})(q[[tz\in\F_q^*]]-1)\\&=-\sum_{t\in T}\sum_{z\in E^{*}}\psi_{E}(uz^{-1})+q\sum_{c\in \F_q^{*}}\sum_{t\in T}\psi_{\F_q}(c\tr(ut))\\&=|T|+q^2\cdot\#\{t\in T:\,ut\in\F_q^*\}-q|T|.
	\end{align*}
Therefore, \eqref{eqn_II_psiDcond_HtoW} is equivalent to the condition that $\#\{\lambda\in T:\,u\lambda\in\F_{q}^{*}\}=1$ or $0$ according as $u^{-1}$ is in $\mathcal{T}$ or not. The claim now follows.
\end{proof}

In Section \ref{sec_extField}, we will see that Lemma \ref{lem_II_HtoW_qodd} implies that there is no transitive $m$-ovoids in Theorem \ref{thm_II_HtoWcons} (1). We shall also need the following two results in Section \ref{sec_extField}. Let $\Gamma^{\#}(V',\kappa')$ be as defined in \eqref{eqn_GamV'def}.
\begin{lemma}\label{lem_HtoW_Tbound}
Take notation as above, and assume that there is a group $H\le\Gamma^{\#}(V',\kappa')$ that is transitive on $\cM$. Then $|T|$ divides $bf/2$ or $bf$ according as $q$ is even or odd. Moreover, if $q$ is odd and $|T|$ does not divide $\frac{1}{2}bf$, then $\mathcal{T}$ is $\F_q^*$-invariant.
\end{lemma}
\begin{proof}
We only consider the case where $q$ is odd for clarity, and the other case is similar. Since $H$ is transitive on $\cM$, its induced action on $X:=\{\square\cdot H'(v):\,\la v\ra\in\cM\}=\{\square\cdot t: \,t\in T\}$ is transitive. Let $\tilde{H}$ be the corresponding group induced on $X$. Then it is a subgroup of $C_2\times C_{bf/2}$, where  $C_2$ is contributed by the similarities in $H$ and  $C_{bf/2}$ is induced from the Galois group of $\F_{q^b}$. Since $X$ is a single orbit of $\tilde{H}$ and $|X|=|T|$, we deduce that $|T|$ divides $bf$.

Now further suppose that $|T|$ does not divide $\frac{1}{2}bf$. Then $|\tilde{H}|$ does not divide $\frac{1}{2}bf$ either. By considering the projection to the second factor $C_{bf/2}$, we deduce that $\tilde{H}$ contains the first factor $C_2$. Therefore, there is an element $g\in H$ such that $\lambda(g)$ is a nonsquare of $\F_{q^{b/2}}^*$, cf. \eqref{eqn_Gam}. For each $t\in T$, the union of $\square\cdot t$ and $\square\cdot\lambda(g)t$ is $\F_q^*\cdot t$. We thus deduce that $\mathcal{T}$ is the union of $\F_q^*$-cosets as desired.
\end{proof}

\begin{thm}\label{thm_II_new}
Take notation as above and assume that $b=6$. If $\mathcal{T}$ is $\F_q^*$-invariant, then we have $|T|\ge q$ or $2q$ according as $q$ is even or odd.
\end{thm}
\begin{proof}
We only consider the case $q$ even for brevity, and the other case is similar. In this case, $\square=\F_q^*$. Set $E=\F_{q^3}$, and write $\tr:=\tr_{q^3/q}$. Take a complete set $R$ of $\F_q^*$-coset representatives of $\F_{q^3}^*$ such that $T=R\cap \mathcal{T}$. By the same calculation as in the proof of Theorem \ref{thm_II_HtoWcons}, the left hand side of \eqref{eqn_II_psiDcond_HtoW} equals
\begin{align*}
\sum_{t\in T}&\sum_{y\in\F_{q}^{*}}K(\psi_{E},yt,u)=-\sum_{t\in T}\sum_{z\in E^{*}}\psi_{E}(uz^{-1})+q\sum_{c\in\F_q^*}\sum_{z\in R}\sum_{t\in T}\psi_{\F_q}(c\tr(uz^{-1})) [[ \tr(tz)=0]]\\
&=|T|+q\sum_{c\in\F_q}\sum_{(z,t)\in R\times T}\psi_{\F_q}(c\tr(uz^{-1})) [[ \tr(tz)=0]]-q\sum_{(z,t)\in R\times T}[[ \tr(tz)=0]]\\
&=|T|+q^2\cdot\#\{(z,t)\in  R\times T:\, \tr(uz^{q+q^2})=0, \tr(tz)=0\}-q(q+1)|T|.
\end{align*}
In the last equality, we used the fact $\tr(uz^{q+q^2})=z^{1+q+q^2}\tr(uz^{-1})$. By \eqref{eqn_II_psiDcond_HtoW}, we deduce that $\#\{(z,t)\in  R\times T:\, \tr(uz^{q+q^2})=0, \tr(tz)=0 \}$ should be equal to $q+|T|$ or $|T|$ according as $u\in \mathcal{T}$ or not. Fix an element $u\in\mathcal{T}$. For each $t\in T$, $\#\{z\in R:\, \tr(uz^{q+q^2})=0, \tr(tz)=0\}$ is the intersection size of a conic and a line and thus is at most $2$. We deduce that  $q+|T|\le 2|T|$, i.e., $|T|\ge q$ as desired. This completes the proof.
\end{proof}

\section{Proof of the main results, I}
\label{sec:mainresults}

This and the next sections are devoted to the proof of Theorems \ref{thm_PGamUTransi}-\ref{thm_PGamSpTransi}. Let $V$ be a vector space of dimension $d$ over $\F_q$ equipped with a nondegenerate reflexive sesquilinear form or a quadratic form $\kappa$, and suppose that $d-1$ and $q$ are not both even in the case $\bO$. Let $\cP$ be the associated polar space and assume that its rank $r$ is at least $2$. Let $\theta_r$ be the ovoid number as in Table \ref{tab_valuede}. Let $e$ be the smallest integer such that $\theta_r|q^e-1$, which is listed in the last column of Table \ref{tab_valuede}. We observe that $\theta_r=p^{ef/2}+1$ in all six cases, where $q=p^f$ with $p$ prime. By the definition of $m$-ovoids, we need $r\ge 2$. We assume that
\[
e>2,\quad (p,ef)\ne(2,6),\
\]
so that $\Phi_{ef}^*(p)>1$ by Lemma \ref{lem_zsigmondy}.

\begin{remark}\label{rem_ege2Explain}
	By Table \ref{tab_valuede}, the condition $e>2$ only excludes $Q^+(3,q)$ from consideration. It is a grid, which is lack of geometric significance. The condition $(p,ef)\ne(2,6)$ further excludes $H(3,4)$, $W(5,2)$, $Q^-(5,2)$, $Q^+(7,2)$ from consideration, cf. Remark \ref{rem_end_sec4} (1).
\end{remark}

Let $\cM$ be a proper $m$-ovoid of $\cP$ with an automorphism group $H_0\le\textup{P}\Gamma(V,\kappa)$ such that $H_0$ is transitive on $\cM$ and its full preimage $H$ in $\Gamma(V,\kappa)$ is irreducible on $V$. Set
\[
G=H\cap \GL(V),\quad\overline{G}:=G/G\cap Z(\GL(V)).
\]
Since $H_0$ is transitive on $\cM$ and $\cM$ has size $m\theta_r=m(p^{ef/2}+1)$, we deduce that $\theta_r$ divides $|H_0|$. In particular, $\Phi_{ef}^*(p)$ divides $|H|$. Since $\gcd(ef,\Phi_{ef}^*(p))=1$ by Section \ref{sec_primdiv} and $[G:\,H]$ divides $ef$, we deduce that $\Phi_{ef}^*(p)$ divides $|G|$. From Table \ref{tab_valuede}, we check that $df/2<ef\le df$ holds in each case. Therefore, Theorem 3.1 of \cite{bamberg2008overgroups} is applicable to the group $G$, which classifies the subgroups of $\textup{GL}_d(q)$ whose orders are divisible by $\Phi_{e'f}^*(p)>1$ for some integer $e'$ such that $d/2<e'\le d$.\medskip

In the sequel, we follow the line of \cite[Theorem 3.1]{bamberg2008overgroups} to prove Theorems \ref{thm_PGamUTransi}-\ref{thm_PGamSpTransi} together. The extension field case is dealt with separately in Section \ref{sec_extField}, and the remaining cases are handled in this section.

\subsection{The classical examples}
\label{subsec_CL}
There are two cases as follows:
\begin{enumerate}
	\item The group $G$ satisfies that $G^{(\infty)}=\Omega(V,\kappa)$. Since $e>2$, we have $d\ge 2,3,4$ in the cases $\bS$, $\bU$, $\bO$ respectively, and so $\Omega(V,\kappa)$ is transitive on the nonzero singular vectors by \cite[Lemma 2.10.5]{kleidman1990subgroup}. In this case, there is no proper transitive $m$-ovoid.
	\item $\cP=W(d-1,q)$, and $G^{(\infty)}=H^{(\infty)}=\Omega^{\pm}_{d}(q)$, where $d$ and $q$ are even. In this case, $\Gamma \textup{O}^{+}_{d}(q)$ is not divisible by $\Phi_{ef}^*(p)$ upon direct check, so we must have $H^{(\infty)}=\Omega^{-}_{d}(q)$. Let $Q$ be a nondegenerate $H^{(\infty)}$-invariant quadratic form which is unique up to a scalar. Since $q$ is even, we have $\cP\setminus Q^{-}(d-1,q)=\{\la v\ra:\,Q(v)=1\}$, where $Q$ is the corresponding quadratic form of $Q^{-}(d-1,q)$. Again by \cite[Lemma 2.10.5]{kleidman1990subgroup}, $\Omega^{-}_{d}(q)$ is transitive on both $Q^{-}(d-1,q)$ and its complement, and they form  $\frac{q^{d/2-1}-1}{q-1}$- and $q^{d/2-1}$-ovoids of $W(d-1,q)$ respectively. This is Example \ref{exa_Classical}.
\end{enumerate}

\subsection{The imprimitive examples}
In this case, we must have $q=p$, $\Phi_{e}^{*}(p)=e+1$, and $G=H$ stabilizes a direct sum decomposition $ \mathcal{D}:\,V=U_{1}\oplus\ldots\oplus U_{d}$, where each $U_{i}$ has dimension $1$ and is nondegenerate. Moreover, $H$ is a subgroup of $\GL_{1}(q)\wr S_{d}$ in its product action and induces a primitive group on the factors $\{U_{1},\ldots,U_{d}\}$. The possible values of $q$, $e$ and $d$ are listed in Table \ref{tab_imp_qde}.
\begin{table}[!h]
	\begin{center}
		\caption{The feasible $(q,e,d)$'s in the imprimitive case}
		\label{tab_imp_qde}
		\begin{tabular}{ccc|ccc|ccc}
			\toprule
			$q$ & $e$ & $d$ & $q$ & $e$ & $d$ & $q$ & $e$ & $d$ \\
			\hline
			2 & 4 & 5,6,7 &2 & 18 & 19,\ldots,35 &3 &6 & 7,\ldots,11  \\
			2 & 10 & 11,\ldots,19 &  3 &4 & 5,6,7 & 5 &6 & 7,\ldots,11\\
			2 & 12 & 13,\ldots,23 &  &  & & &&\\
			\bottomrule
		\end{tabular}
	\end{center}
\end{table}	
Since $q$ is prime, $d>2$ and $\dim(U_i)=1$, we must be in the case $\bO$. Moreover, since each $U_i$ is nondegenerate, $q$ must be odd.  The group $H$ consists of similarities and it permutes the $U_i$'s transitively, so the $U_i$'s are isometric. We choose a proper vector $v_i$ from each $U_i$ such that $Q(v_i)=\lambda$ for some constant $\lambda\in\F_q^*$. Then $Q(x)=\lambda\sum_{i=1}^dx_i^2$, where $x=\sum_{i=1}^dx_iv_i$.
For each possible pair $(q,d)$, we determine the type of $Q$ and the value of $e$ by Table \ref{tab_valuede}. If $(q,d,e)$ is not in Table \ref{tab_imp_qde}, then the corresponding case can not occur.  The remaining possible cases are $Q(4,3)$, $Q^{+}(5,3)$, $Q(6,3)$, $Q(6,5)$, $Q^{+}(7,3)$ and $Q^{+}(7,5)$. In each case, we realize $\textup{PGO}^\epsilon_d(q)$ as the stabilizer of the corresponding quadric in $\PGL_{d}(q)$, and compute the stabilizer $N$ of $\{\la v_{i}\ra:\,1\leq i\leq6\}$ in $\textup{PGO}^\epsilon_d(q)$ by Magma \cite{Magma}.  In each case, we enumerate all the irreducible subgroups of $N$ whose order is divisible by $\theta_r$ and check whether any of their orbits form an $m$-ovoid. This leads to Example \ref{exa_Imprimitive}.\medskip

\subsection{The symplectic type examples}
In this case, $q=p$, $H=G$ normalizes an extraspecial $2$-group, and we have one of the following cases:
\begin{itemize}
	\item[(a)] $p=3$, $e=d=4$, and $G\leq(2_{-}^{1+4}. \textup{O}_{4}^{-}(2))\circ 2$;
	\item[(b)] $p=3$, $d=8$, $e=6$ and $G\leq(2_{+}^{1+6}.\textup{O}_{6}^{+}(2))\circ 2$;
	\item[(c)] $p=5$, $d=8$, $e=6$ and either $H\leq((4\circ2^{1+6}). \textup{Sp}_{6}(2))\circ 4$ or $G\leq(2_{+}^{1+6}.\textup{O}_{6}^{+}(2))\circ 4$.
\end{itemize}
In cases (a) and (b), the term $\circ 2$ can be omitted since it refers to the group of scalar multiples by $\F_3^*$ and is already contained in the center of the extraspecial $2$-groups. We assume that $q=p=3$. Take $D_{8}=\langle a,b:\,a^4=b^2=1,\,b^{-1}ab=a^{-1}\rangle$, $Q_{8}=\langle x,y:\,x^2=y^2,\,x^4=1,\,x^{-1}yx=y^{-1}\rangle$. Let $\rho_D$ and $\rho_Q$ be their representations as follows:
\begin{align*}
	\rho_D:&\,D_8\rightarrow \GL_2(3),
	\,a\mapsto \left( \begin{array}{cc} 0 & 1 \\-1 & 0 \\\end{array}\right),
	\, b\mapsto \left(\begin{array}{cc}                0 & 1 \\1 & 0 \\\end{array}
	\right),\\
	\rho_Q:&\,Q_8\rightarrow \GL_2(3),\,
	x\mapsto \left( \begin{array}{cc}0 & 1 \\ -1 & 0 \\\end{array}
	\right),\,y\mapsto \left(\begin{array}{cc}         1 & 1 \\1 & -1 \\
	\end{array}\right).
\end{align*}
The representation $\rho_D$ preserves the symmetric bilinear form $f_D(u,v)=u_1v_1+u_2v_2$, and $\rho_Q$ preserves the alternating bilinear form $f_Q(u,v)=u_1v_2-u_2v_1$.

In the case (a), the extraspecial 2-group $E=2_{-}^{1+4}=D_{8}\circ Q_{8}$, $V=\F_{3}^{2}\otimes\F_{3}^{2}$, and $E$ acts on $V$ via the representation $\rho_1:=\rho_D\otimes\rho_Q$. The group $H$ lies in the normalizer $N$ of $\rho_1(D)$ in $\GL_4(3)$ and preserves the alternating form $f_1:=f_D\otimes f_Q$ on $V$. Set $e_{0}=(1,0)$, $e_{1}=(0,1)$, and $e_{ij}=e_{i}\otimes e_{j}$ for $i,j\in\{0,\,1\}$. With respect to the basis $e_{00}$, $e_{01}$, $e_{10}$, $e_{11}$, we can compute the matrix form of the representation $\rho_1$.  By Magma \cite{Magma}, we compute the normalizer $N$ of $\rho_1(E)$ in $\GL_{4}(3)$ and see that there are $10$ subgroups of $N$ with order divisible by $3^2+1$. Six of them are transitive on points of $\cP$. Among the remaining $4$ subgroups of $N$, exactly one of order $120$ (i.e., $2.\Omega_4^-(2)$) gives rise to transitive $m$-ovoids. This group has two orbits on points of $W(3,3)$ and each is a $2$-ovoid of $W(3,3)$. This is Example \ref{exa_Symplectic}.

In the case (b), the extraspecial 2-group $E=2_{+}^{1+6}=D_{8}\circ D_{8}\circ D_{8}$, $V=\F_3^2\otimes\F_3^2\otimes\F_3^2$, $E$ acts on $V$ via the representation $\rho_2:=\rho_D\otimes\rho_D\otimes\rho_D$ and preserves the symmetric bilinear form $f_2=f_D\otimes f_D\otimes f_D$. The group $H$ lies in $N=N_{\GL_8(3)}(\rho_2(E))=2_{+}^{1+6}\cdot \textup{O}_{6}^{+}(2)$, and the polar space is $\cP=Q^+(7,3)$. Set $e_{0}=(1,0)$, $e_{1}=(0,1)$, and $e_{ijk}=e_{i}\otimes e_{j}\otimes e_k$ for $i,j,k\in\{0,\,1\}$. We can similarly work out the matrix form of $\rho_2$ with respect to this basis, and compute $N$ as the normaliser of $\rho_2(E)$ in $\GL_8(3)$ in Magma. In total, there are $65$ subgroups of $N$ having orders divisible by $3^3+1$. Among these groups, there are $5$ transitive subgroups; none of the remaining $60$ subgroups gives rise to transitive $m$-ovoids of $Q^{+}(7,3)$.

In the case (c), we are in case $\bL$ if $G\leq((4\circ2^{1+6}). \textup{Sp}_{6}(2))\circ 4$, and in case $\bO^+$ if $G\leq (2_{+}^{1+6}.\textup{O}_{6}^{+}(2))\circ 4$ by \cite[Table 4.6.B]{kleidman1990subgroup}. We do not consider the case $\bL$, so we only study the latter case. By Magma \cite{Magma}, there are $11$ subgroups of $N_{GL_{8}(5)}(\rho_2(E))$ whose order is divisible by $5^3+1$, where $\rho_2$ is defined similarly as in case (b). None of them gives rise to transitive $m$-ovoids of $Q^{+}(7,5)$.

\subsection{The nearly simple cases}\label{sec_nearlysimple}

In this case, $S\leq\overline{G}\leq\textup{Aut}(S)$, where $S$ is a finite non-abelian simple group and a covering group $\tilde{S}$ of $S$ acts on $V$ absolutely irreducibly, cf. Definition 2.1.3 and Theorem 2.2.19 of \cite{bray2013maximal}. Let $\rho$ be the representation of $\tilde{S}$ afforded by $V$.  The relevant information on the Frobenius-Schur indicator of $\rho$ is available from  \cite{Hissrepreofquasi2001} and \cite{LubeckSmalldegree}. In the cases where the Frobenius-Schur indicator is $\circ$, we use Lemma \ref{lem_FScircDet} to determine whether it is of case $\bL$ or $\bU$, and we do not consider the cases that are of type $\bL$. In the cases where the Frobenius-Schur indicator is $+$, we do not need to determine the actual type of the quadratic space. The expected type $\bO^\epsilon$ can be read off from the relation between $d$ and $e$ in Table \ref{tab_valuede}, and in most cases it works to exclude the existence of transitive $m$-ovoids by showing that the ovoid number does not divide $|\Aut(S)|$. For the cases that survive the divisibility test, the explicit matrix form of $\rho$ is  available  either in Magma \cite{Magma} or in Atlas \cite{Atlas}, and the Magma command \textit{ClassicalForms} returns the associated classical form $\kappa$. At this point, we do an exhaustive search and are able to determine all the transitive $m$-ovoids in most cases.

By \cite[Theorem 3.1]{bamberg2008overgroups}, there are four families to consider, which we handle separately below. There are cases where $S$ is transitive on $\cP$, cf. \cite{giudici2020subgroups}, and we exclude them from consideration.

(I) \textit{Alternating group cases}.  We first consider the fully deleted module examples. By Theorem 3.1 of \cite{bamberg2008overgroups}, we have $q=p$, and $p^e=2^4,\,2^{10},\,2^{12},\,2^{18},\,3^{4},\,3^{6}$ or $5^6$,
$A_n\le G\le S_n\times Z$. Here, $Z$ is the center of $\GL(V)$, and $e$ is as defined in Table \ref{tab_valuede}. By the same table, we see that $d$ should be one of $e,e+1,e+2$. The type of the associated polar spaces are available in Section \ref{subsec_FDM}. We check with Magma that the cases with proper transitive $m$-ovoids are exactly those listed in Table \ref{tab_fully}. We then consider the other absolutely irreducible representations of $A_n$ or its covering groups. We check with Magma that the cases with proper transitive $m$-ovoids are exactly those listed in Table \ref{tab_otherAlt}.  The transitive $m$-ovoids in the Alternating group case are summarized in Example \ref{exa_Alternating}.

(II) \textit{Sporadic simple group cases}. After the same analysis as above and computer search by Magma, it turns out that the only case with transitive $m$-ovoids is the one in Example \ref{exa_thmsporadic}.

(III) \textit{Cross-characteristic cases}. In this case, $S$ is simple group of Lie type whose defining characteristic is distinct from $p$.  The cases with transitive $m$-ovoids are exactly those listed in Table \ref{tab_thmcross}.

(IV) \textit{Natural-characteristic cases}. In this case, $S$ is a simple group of Lie type whose defining characteristic is $p$. The Frobenius-Schur indicators are available from Propositions 5.7.1-5.7.4 and Table 5.6 of \cite{bray2013maximal}, and they uniquely determine the type of $\cP$. We exclude the cases  $(S,\cP)=(\PSL_3(q^2),H(8,q^2))$ with $q\not\equiv 2\pmod 3$ and $(S,\cP)=(\SL_{3}(q^2),H(8,q))$ with $q\equiv 2(\textup{mod }3)$ by the fact that the ovoid number of $\cP$ does not divide $\Aut(S)$, see also  \cite[Theorem 4.1]{bamberg2008overgroups}. There are five cases where $S$ is transitive on $\cP$, cf. \cite{giudici2020subgroups}. For the remaining five cases,  there are transitive $m$-ovoids in each case. Please refer to Example \ref{exa_thmNatural} for details.

\begin{remark} Here are comments on the relevant results in \cite{bamberg2008overgroups}.
\begin{enumerate}
\item[(1)] In the Alternating group case, the two cases $(G^{(\infty)},\cP)=(2^{\cdot}A_7,H(3,3^2))$ and $(2^{\cdot}A_7,H(3,5^2))$ are missing from \cite[Theorem 4.2]{bamberg2008overgroups}.
\item[(2)] In the Sporadic group case, the case where $G'=2^\cdot M_{22}$, $d=e=10$, $q=2$ does not occur but is erroneously listed in \cite[Theorem 3.1]{bamberg2008overgroups}, cf. \cite[Table 4.4]{bray2013maximal}. Also, the cases $(S,\cP)=(J_3,\,H(8,4))$, $(M_{22},\,H(5,4))$ are missing from Theorems 4.1 and 4.2 of \cite{bamberg2008overgroups} respectively.
\item[(3)] In the Cross-characteristic case, the following two lines in \cite[Theorem 3.1]{bamberg2008overgroups} should be removed due to the fact that the unique $d$-dimensional absolutely irreducible representation of $S$ over $\F_4$ can not be realized over $\F_2$: (1) $(q,d,e)=(2,18,18)$ and $S=\PSL(2,37)$; (2) $(q,d,e)=(2,12,12)$ and $S=\PSp(4,5)$. Also, the following cases are missing from \cite[Theorem 4.1]{bamberg2008overgroups}:
	\begin{align*}
		(S,\cP)=&(\PSL_{2}(7), H(2,9)), \,(\PSL_{2}(7), H(2,25)),\, (\PSL_{2}(11), H(4,4)), \,(\PSL_{2}(19), H(8,4)).
	\end{align*}
\item[(4)] In the Natural-characteristic case, the following two cases in \cite[Theorem 3.1]{bamberg2008overgroups} should be removed due to the fact that $\Phi^*_{ef}(p)$ does not divide $|\Aut(S)|$ in either case: (a) $(q,d,e)=(3^f,7,6)$ with $f$ even and $S={}^{2}G_{2}(q^{1/2})$, (b) $(q,d,e)=(2^f,4,4)$ with $f$ even and $S=\textup{Sz}(q^{1/2})$.
\end{enumerate}
\end{remark}

\section{Proof of the main results, II}\label{sec_extField}
This section is devoted to the proof of the main results in the extension field case. We keep the same notation as introduced in the beginning of Section \ref{sec:mainresults}. In particular, $V$ is a $d$-dimensional vector space over $\F_q$ equipped with a nondegenerate form $\kappa$, and $\cP$ is the polar space of rank $r\ge 2$ associated with $(V,\kappa)$. Let $\cM$ be a proper $m$-ovoid of $\cP$ with an automorphism group $H_0\le\textup{P}\Gamma(V,\kappa)$ such that $H_0$ is transitive on $\cM$ and its full preimage $H$ in $\Gamma(V,\kappa)$ is irreducible on $V$, and $G=H\cap \GL(V)$.\medskip

In this section, we suppose that there is a divisor $b>1$ of $d$ such that we can regard $V$ as a vector space $V'$ of dimension $d/b$ over $\F_{q^b}$ and $H$ preserves a nondegenerate reflexive sesquilinear form or quadratic form $\kappa'$ on $V'$. \textbf{We take $b$ maximal with respect to this property}. Let $\cP'$ be the polar space associated  to $(V',\kappa')$. We refer to \cite[Table 4.3.A]{kleidman1990subgroup} for the connection between $\kappa$ and $\kappa'$, see also Table \ref{tab_extfieldP'gt0}. We have $H\le \Gamma^{\#}(V',\kappa')$, where $\Gamma^{\#}(V',\kappa')$ is as defined in \eqref{eqn_GamV'def}.

By \cite[Theorem 3.1]{bamberg2008overgroups}, there are two subcases according as whether $\Phi_{ef}^*(p)$ is relatively prime to $b$ or not. If not, then $q=p$, $\Phi_{e}^*(p)=b=d=e+1$, $H=G\leq \textup{GL}_{1}(q^d)\cdot d$, and $p^e=2^4,2^{10},2^{12},2^{18},3^4,3^6,5^6$. In all these cases, we have $d-e=1$ and $q=p$, so by the relation between $d$ and $e$ in Table \ref{tab_valuede} the possible cases are all in case $\bO$ with $p^e=3^4,3^6$ or $5^6$. Since $\theta_{r}=p^{e/2}+1$ divides $|\textup{GL}_{1}(q^d)\cdot d|=2(p-1)d$, we exclude the case $p^e=5^6$. In the remaining two cases, we check that there is no transitive $m$-ovoid of $Q(d-1,p)$ by Magma \cite{Magma}.

\textit{From now on, we consider the case where $\Phi_{ef}^*(p)$ is relatively prime to $b$.} The possible pairs of $(\cP,\cP')$ are listed in Table \ref{tab_extfieldP'gt0}, which are obtained by iteratively applying field reductions in \cite[Table 4.3A]{kleidman1990subgroup}. There are some cases that can be excluded by divisibility conditions.
\begin{lemma}\label{lem_extfld2468}
	There is no transitive $m$-ovoid in the cases 2,4,6,8,12 of Table \ref{tab_extfieldP'gt0} with a transitive automorphism group contained in $\textup{P}\Gamma^\#(V',\kappa')$.
\end{lemma}
\begin{proof}
	The case 12 is not possible, since we need $e\le d-2$ by \cite[Theorem 3.1]{bamberg2008overgroups} but we have $e=d$ by Table \ref{tab_valuede}.  The proofs for the other four cases are similar, so we only prove case 4. In this case, $e=d-1$ by Table \ref{tab_valuede}, $H\leq\GamO_{d/b}(q^b)$, and the order of $\GamO_{d/b}(q^b)$ is $N:=q^{b(d/b-1)^2/4}(q^b-1)\Pi_{i=1}^{(d/b-1)/2}(q^{2bi}-1)bf$. The ovoid number of $\cP=Q(d-1,q)$ is $\theta_{(d-1)/2}=q^{(d-1)/2}+1$. We claim that $\theta_{(d-1)/2}$ does not divide $N$, and the result then follows. Recall that we assume $\gcd(b,\Phi_{ef}^*(p))=1$.
	Let $r_{0}$ be a prime divisor of $\Phi_{ef}^*(p)$. Then $r_{0}$ divides $q^{(d-1)/2}+1$ and  is relatively prime to $q^{k_{0}}-1$ if $1\leq k_{0}\leq d-1$. In addition, $r_{0}$ is relatively prime to $f$ by Section \ref{sec_primdiv}, and $\gcd(r_0,b)=1$. We deduce that $r_{0}$ does not divide $|\GamO_{d/b}(q^b)|$ if $b<d$. In the case $d=b$, $r_0$ can not divide $q^d-1$, since otherwise it would divide $\gcd(q^d-1,q^{e}-1)=q-1$: a contradiction to the fact $r_0$ is a primitive prime divisor. Hence $r_{0}$ does not divide $|\GamO_{d/b}(q^b)|$. This completes the proof.
\end{proof}

Throughout this section, we set
\begin{equation}\label{eqn_Kdef}
	K:=H\cap \Delta(V',\kappa').
\end{equation}
\begin{lemma}\label{lem_K_primdiv}
	Let $e$ be determined by $d$ as in Table \ref{tab_valuede}. Then $\Phi_{ef}^*(p)$ divides $|K|$.
\end{lemma}
\begin{proof}
	It is clear that $[H:\,K]$ divides $[\Gamma(V',\kappa'):\,\Delta(V',\kappa')]=bf$.  Since the ovoid number $\theta_r$ divides $|H|$, we have that $\Phi_{ef}^*(p)$ divides $|H|$. Recall that we assumed that $\Phi_{ef}^*(p)$ is relatively prime to $b$, and it is relatively prime to $f$ by Section \ref{sec_primdiv}. We thus deduce that $\Phi_{ef}^*(p)$ divides $|K|$.
\end{proof}

\begin{lemma}\label{lem_Kirr}
	Let $e$ be determined by $d$ as in Table \ref{tab_valuede}, and assume that $d/b\ge 3$. Then $K$ acts on $V'$ irreducibly if either of the following holds: (i) $e=d$, (ii) $b=2$, $d/2$ is odd and $e=d-2$.
\end{lemma}
\begin{proof}
	First consider the case $e=d$. It follows from Lemma \ref{lem_K_primdiv} that $\Phi_{df}^{*}(p)$ divides $|K|$. If $K$ stabilizes a subspace $U'$ of dimension $k$ with $1<k<d/b$, then
	\[
	K\le q^{k(d/b-k)}\cdot(\GL_{k}(q^b)\times\GL_{d/b-k}(q^b)).
	\]
	The latter group has an order not divisible by $\Phi_{df}^{*}(p)$, which is a contradiction. We conclude that $K$ is irreducible on $V'$ in this case.
	
	We next consider the case where $b=2$, $d/2$ is odd and $e=d-2$. In this case, $d\ge 3b=6$, and $\Phi_{ef}^{*}(p)$ divides $|K|$ by Lemma \ref{lem_K_primdiv}. Suppose that $K$ is reducible on $V'$. By the same argument as in the previous case, we deduce that if a proper $K$-invariant subspace of $V'$ has dimension $k$ then $\max\{2k,d-2k\}\ge d-2$. That is, a proper $K$-invariant subspace of $V'$ must have dimension $1$ or $d/2-1$. If $U'$ is a $K$-invariant subspace of dimension $d/2-1$, then $U'^\perp$ is a $K$-invariant subspace of dimension $1$. Therefore, there is always a $K$-invariant subspace $U_1=\F_{q^2}\cdot u$ of dimension $1$ over $\F_{q^2}$. Fix $g\in H$. Since $K\unlhd H$, $U_1^g$ is also $K$-invariant. It follows that $U_1+U_1^g$ is $K$-invariant, and thus has dimension $1$ or $d/2-1$. Since $d/2\ge 3$, we see that it has dimension $1$, i.e., $U_1=U_1^g$. We thus have shown that $U_1$ is $H$-invariant, contradicting  the hypothesis that $H$ is irreducible on $V$. This completes the proof.
\end{proof}

By Table \ref{tab_extfieldP'gt0}, if $\cP=H(d-1,q)$ with $d$ even or $\cP=Q(d-1,q)$, then $\cP'$ must be of the same type and this has been excluded by Lemma \ref{lem_extfld2468}. We consider the remaining cases one by one in the sequel. Recall that $m$-ovoids exist only in finite classical polar spaces of rank at least $2$.\\

\subsection{The case $\cP=H(d-1,q)$, $d\ge 5$ odd}\label{subsec_II_Hdood}

In this case, we must have $\cP'=H(d/b-1,q^b)$ by Table \ref{tab_extfieldP'gt0}. We first consider the case where $b=d$. In this case, $\F_{q^{d/2}}\cap\F_q=\F_{q^{1/2}}$, and  we write $\Gamma^\#=\Gamma^\#(V',\kappa')$ for short. Suppose that $\cM=\la x\ra^H$, and $|\cM|=m(q^{d/2}+1)$. Since $|\cM|$ divides $|\textup{P}\Gamma^\#|=\frac{df(q^{d/2}+1)}{q^{1/2}+1}$, we deduce that $m(q^{1/2}+1)$ divides $df$. The element $g:\,u\mapsto x^{1-q^{d/2}}u^{q^{d/2}}$ has order $2$ and lies in $\Gamma^\#(V',H')$. We can verify that $g(x)=x$. Therefore, $|\cM|\le |\la v\ra^{\Gamma^\#}|\le\frac{df(q^{d/2}+1)}{2(q^{1/2}+1)}$.
Together with the bound on $m$ in Theorem \ref{thm_Bboundm}, we have
\[
(q^{d/2}+1)\cdot\frac{-3+\sqrt{9+4q^{d/2}}}{2(q-1)}\le \frac{df(q^{d/2}+1)(q^{1/2}-1)}{2(q-1)},
\]
which yields $-3+2p^{df/4}\le df(q^{1/2}-1)\le df(p^{df/6}-1)$. Set $y:=df/2$ which is an integer such that $y\ge 3$. We deduce that
\[
2p^{y/2}\le 2y-3+2p^{y/2}\le 2yp^{y/3},\quad\textup{i.e., } p^{y/6}\le y.
\]
Together with the condition that  $p^{f/2}+1$ divides $df$,  we check with computer that there is no feasible $(p,b,f)$ tuple. This concludes the discussion for the case $d=b$.\\

For the rest of this subsection we assume that $b<d$. By Table \ref{tab_valuede}, we have $e=d$. By Lemma \ref{lem_K_primdiv} and Lemma \ref{lem_Kirr},  $\Phi_{ef}^*(p)$ divides $|K|$ and $K$ is irreducible on $V'$. We observe that $q^b=p^{bf}$ with $f$ even and $bf>2$, and $K$ is not an extension field subgroup of $\textup{GL}(V')$ by the choice of $b$. We apply \cite[Theorem 4.1]{bamberg2008overgroups} to $K\le\textup{GL}(V')$ and see that $K^{(\infty)}=\textup{SU}_{d/b}(q^{b/2})$. In this case, $H^{(\infty)}=\textup{SU}_{d/b}(q^{b/2})$. We take the same notation as in Section \ref{subsec_Hdodd}. By \cite[Lemma 2.10.5]{kleidman1990subgroup}, $\textup{SU}_{d/b}(q^{b/2})$ acts transitively on the nonzero vectors of $V'$ that has the same $H'$-values. In particular, it acts transitively on $\cM_1:=\{\la v\ra_{\F_q}:\,H'(v)=0\}$, and $\cM_1$ is a  $\frac{q^{(d-b)/2}-1}{q-1}$-ovoid of $H(d-1,q)$ by \cite{KellyCons}.
We now consider the case where $\cM$ is a subset of $\cP\setminus\cM_1$. Since $\cM$ is the union of $\SU_{d/b}(q^{b/2})$-orbits, it is of the form as in Section \ref{subsec_Hdodd}, i.e., $\cM=\{\la v\ra_{\F_q}:\,H'(x)\in \mathcal{T}\}$ with $\mathcal{T}=\F_{q^{1/2}}^*\cdot T$, where $\tr_{q^{b/2}/q^{1/2}}(t)=0$ for each $t\in\mathcal{T}$. Since $H_0$ is transitive on $\cM$, its induced action on the set $\{\F_{q^{1/2}}^*\cdot H'(v):\,\la v\ra\in \cM\}=\{\F_{q^{1/2}}^*\cdot t:\,t\in T\}$ is transitive. We deduce that $|T|$ divides $bf/2$. By Lemma \ref{lem_Herm_Msize}, $\cM$ is an $m$-ovoid with  $m=\frac{|T|}{q^{1/2}+1}q^{(d-b)/2}$. It follows that $q^{1/2}+1$ divides $|T|$. By Proposition \ref{prop_Herm_Tbound}, we have $|T|\ge\frac{q^{b/2}}{(q^{1/2}-1)(2q^{b/4}+1)}$.
Together with the facts that $q^{1/2}+1$ divides $|T|$ and $|T|$ divides $bf/2$, we deduce that there are only $3$ feasible $(p,b,f,|T|)$ tuples.  We check the necessary and sufficient condition \eqref{eqn_HtoH_psiDcond} by Magma \cite{Magma} for each possible subset $T$. There are exactly two feasible tuples and the corresponding transitive $m$-ovoids are those in Example \ref{exa_FR_HtoH_smallpara}.

\subsection{The case $\cP=Q^-(d-1,q)$, $d\ge 6$}\label{subsec_II_QM}

In this case, we have $e=d$ by Table \ref{tab_valuede}, and $\Phi_{df}^*(p)$ divides $|K|$ by Lemma \ref{lem_K_primdiv}. As in the case 12 of Lemma \ref{lem_extfld2468}, we can show that $H$ does not preserve an $\F_{q^{b'}}$-linear quadratic form of type $\circ$ with $b'>1$ by divisibility considerations. By Table \ref{tab_extfieldP'gt0}, we see that there are two possibilities:
\begin{enumerate}
	\item[(a)] $d/b$ even,  and $\cP'=Q^-(d/b-1,q^b)$, $\kappa'=Q'$, $\kappa=Q=\tr\circ Q'$,
	\item[(b)] $b$ even, $d/b$ odd, $\cP'=H(d/b-1,q^{b})$, $H'(x)=\kappa'(x,x)$, $\kappa=Q=\tr_{q^{d/2}/q}\circ H'$.
\end{enumerate}

Suppose that we are in the case (a). If $d/b=2$, then $H$ preserves a $1$-dimensional Hermitian form over $\F_{q^d}$ by \cite[p.120]{kleidman1990subgroup}, contradicting the maximality of $b$. Hence we have $d/b\ge 4$. By Lemma \ref{lem_Kirr}, $K$ is irreducible on $V'$, where $K$ is defined in \eqref{eqn_Kdef}. We observe that $q^b$ is a square, and $K$ is not an extension field subgroup of $\textup{GL}(V')$ by the choice of $b$.  We apply \cite[Theorem 3.1]{bamberg2008overgroups} to $K\le\textup{GL}(V')$ and see that $K^{(\infty)}=\Omega_{d/b}^-(q^b)$. It follows that $H^{(\infty)}=\Omega_{d/b}^-(q^b)$.
The group $\Omega_{d/b}^-(q^b)$ is transitive on the nonzero vectors of the same $Q'$-value by \cite[Lemma 2.10.5]{kleidman1990subgroup}; in particular, it is transitive on
$\cM_1:=\{\left\langle \lambda v\right\rangle_{\F_q}:\,Q'(v)=0\}$,
and $\cM_1$ is a $\frac{q^{d/2-b}-1}{q-1}$-ovoid of $\cP$ by \cite{KellyCons}.
We next consider the case $\cM\subseteq\cP\setminus\cM_1$, which must be of the form in Section \ref{subsec_QMdbeven}. If $b=2$, then $\cM=\cP\setminus\cM_1$ by Remark \ref{rem_I_QMtoQM_b2} and it is indeed a transitive $q^{d/2-2}$-ovoid. We assume that $b>2$ in the sequel, and take the same notation as in Section \ref{subsec_QMdbeven}. In particular, $\cM=\{\la v\ra_{\F_q}:\,Q'(x)\in \mathcal{T}\}$ with $\mathcal{T}=\square\cdot T$, where the elements of $\mathcal{T}$  have relative trace $0$. Set $\epsilon=1/2$ or $1$ according as $q$ is odd or even. Since $H$ is transitive on $\cM$, its induced action on $\{\square\cdot Q'(v):\,\la v\ra\in\cM\}=\{\square\cdot t:\,t\in T\}$ is also transitive and so $|T|$ divides $|\F_q^*/\square|\cdot bf$. The factor $|\F_q^*/\square|=\epsilon^{-1}$ comes from the similarity group.  By Lemma \ref{lem_QMinus_Msize}, $\cM$ is an $m$-ovoid with  $m=\epsilon|T|q^{d/2-b}$. By Proposition \ref{prop_Ellip_Tbound}, we have  $|T|\epsilon\ge\frac{q^{b}}{(q-1)(2q^{b/2}+1)}$. Together with the fact that $|T|\epsilon$ is a divisor of $bf$, we get a short list of feasible $(p,b,f,|T|)$ tuples. We check the necessary and sufficient condition \eqref{eqn_QMtoQM_psiDcond} with Magma for each possible subset $T$. The resulting transitive $m$-ovoids are exactly the ones in Example \ref{exa_FR_QMtoQM_smallpara}.\\

Suppose that we are in the case (b). In this case, $b$ is even and $d/b$ is odd. First consider the case $b=d$. Suppose that $\cM=\la v\ra^H$ for some $\la v\ra\in\cP$ and $|\cM|=m(q^{d/2}+1)$. Write $\Gamma^\#:=\Gamma\textup{U}^\#_1(q^d)$, let $\Delta^\#$ be its subgroup of similarities, and use $\textup{P}\Gamma^\#$ and $\textup{P}\Delta^\#$ for their projective analogues. The group $\textup{P}\Delta^\#$ is cyclic of order $q^{d/2}+1$ and acts semiregularly on $\cP$, and $\textup{P}\Gamma^\#$ has order $df(q^{d/2}+1)$. We deduce that $m$ divides $df$ by the fact $|\cM|$ divides $|\textup{P}\Gamma^\#|$. The group $\textup{P}\Gamma^\#$ has an involution that stabilizes $\cM$ as in the case $d=b$ of Section \ref{subsec_II_Hdood}, so $|\la v\ra^{\Gamma^\#}|\le \frac{df}{2}(q^{d/2}+1)$. We deduce that $m\le df/2$ by the fact that $|\la v\ra^H|\le |\la v\ra^{\Gamma^\#}|$. Together with the bound on $m$ in Theorem \ref{thm_Bboundm}, we deduce that there are $21$ feasible $(p,b,f)$ tuples and correspondingly $35$ feasible $(p,b,f,m)$ tuples. After applying Theorem \ref{thm_modulareqn}, we have seven remaining tuples: $(2,6,1,3)$, $(2,6,3,6)$, $(2,6,3,9)$, $(2,8,1,4)$, $(3,6,1,2)$, $(3,8,1,4)$, $(5,6,1,3)$.  The tuples $(2,6,1,3)$ and $(2,6,3,9)$ lead to $\cM=\cP$ which is not proper by considering sizes. Let $K$ be the subgroup of $H$ consisting of similarities, cf. \eqref{eqn_Kdef}, and set  $u:=[H:\,K]$ which is a divisor of $df$. We have $|H|=\frac{u}{r}(q^{d/2}+1)$, where $r=:[\textup{P}\Delta^\#:\,K]$. Since $|\cM|$ divides $|H|$, we deduce that $mr$ divides $u$.  For the remaining tuples, we check with Magma that there is no feasible group $H$ that leads to transitive $m$-ovoids. \medskip

We assume that $d/b\ge 3$ in the sequel. By Lemma \ref{lem_Kirr} and Lemma \ref{lem_K_primdiv}, $K$ is irreducible on $V'$ and has order divisible by $\Phi_{df}^*(p)$. By the choice of $b$, $K$ is not an extension field subgroup of $\textup{GL}(V')$. We apply \cite[Theorems 4.1]{bamberg2008overgroups} to $K\le\textup{GL}_{d/b}(q^b)$ and see that either $K^{(\infty)}=\textup{SU}_{d/b}(q^{b/2})$ or it is one of the  cases in Table \ref{tab_Thm4.1Spor}. The group $3^\cdot A_7$ has two orbits on $Q^-(5,5)$, which are both $3$-ovoids. We check with Magma that there are no transitive $m$-ovoids in other cases in the table. This leads to case (c1) of Theorem \ref{thm_PGamOMinusTransi}.

\begin{table}[!h]	
	\begin{center}
		\caption{The sporadic cases in \cite[Theorems 4.1]{bamberg2008overgroups}}
		\label{tab_Thm4.1Spor}
		\begin{tabular}{ccc|ccc}
         \toprule
			$K^{(\infty)}$& $d/b$ & $q^b$ & $K^{(\infty)}$ & $d/b$ & $q^b$  \\ \hline
			$3^\cdot A_7$ & $3$ & $5^2$ &	$3^\cdot J_3$ & $9$ & $2^2$	\\
			$\PSL_2(7)$   & $3$ & $3^2$ &   $\PSL_2(11)$ & $5$ & $2^2$ \\
			$\PSL_2(7)$ & $3$ & $5^2$   &   $\PSL_2(19)$ & $9$ & $2^2$ \\ \bottomrule
		\end{tabular}
	\end{center}
\end{table}

We next consider the case  $K^{(\infty)}=\textup{SU}_{d/b}(q^{b/2})$. In this case, $H^{(\infty)}=\textup{SU}_{d/b}(q^{b/2})$, and it is transitive on the nonzero vectors that have the same $H'$-value. In particular, $\cM_1:=\{\la v\ra_{\F_q}:\,H'(v)=0\}$ is an $H^{(\infty)}$-orbit, and it is a transitive $\frac{q^{(d-b)/2}-1}{q-1}$ ovoid of $\cP=\cQ^-(d-1,q)$ by \cite{KellyCons}. We next consider the case $\cM\subseteq\cP\setminus\cM_1$. In this case, $\cM$ must be of the form in Section \ref{subsec_QMHdbodd}. We take the same notation as therein, and set $\epsilon=1/2$ or $1$ according as $q$ is odd or even. Since $H$ is transitive on $\cM$, its induced action on $\{\square\cdot H'(v):\,\la v\ra\in\cM\}=\{\square\cdot t: \,t\in T\}$ is transitive. We deduce that $|T|$ divides $|\F_q^*/\square|\cdot bf/2$, where $|\F_q^*/\square|=\epsilon^{-1}$ and is contributed by a similarity in the case $q$ is odd. By the bound in Proposition \ref{prop_EllipHerm_Tbound}, if $\cM$ is an $m$-ovoid with $m=\epsilon|T|q^{(d-b)/2}$,  then $|T|\ge\frac{q^{b/2}}{\epsilon(q-1)(2q^{b/4}+1)}$. It follows that
\begin{equation}\label{eqn_II_HtoQM_ine}
	q^{b/2}\le\frac{1}{2}bf(q-1)(2q^{b/4}+1).
\end{equation}
By Remark \ref{rem_I_HtoQM_b24}, we must have $b>2$ and $\cM=\cP\setminus\cM_1$ if $b=4$. It remains to consider the case $b\ge 6$. Write $y:=\frac{1}{2}bf$, which is an integer such that $y\ge 3$. We have $q=p^f\le p^{y/3}$. From \eqref{eqn_II_HtoQM_ine} we deduce that $p^{y/6}<2y$. In particular, $2^{y/6}<2y$ from which we deduce that $y=\frac{1}{2}bf\le 37$. Then we deduce from $p^{y/6}<2y$ that $p\le 31$. There are $47$ feasible $(p,b,f)$ tuples and correspondingly $70$ feasible $(p,b,f,|T|)$ tuples satisfy that $|T|\epsilon$ as a divisor of $\frac{1}{2}bf$ and $|T|\epsilon\ge\frac{q^{b/2}}{(q-1)(2q^{b/4}+1)}$.
Since $d/b-1$ is even, we have $q^{\frac{b}{2}(\frac{d}{b}-1)} \equiv 1 \pmod{q+1}$, i.e.,  $m=\epsilon|T|q^{(d-b)/2} \equiv \epsilon|T| \pmod{q+1}$. By applying Theorem \ref{thm_modulareqn}, we obtain $11$ remaining feasible tuples: \begin{align*}
	&(2, 6, 1, 3),\,(2, 6, 3, 3 ),\,(2, 6, 3, 9), \,(2, 8, 1, 4), \,(2, 10, 2, 10 ),( 2, 10, 3, 15),\,\\ &( 2, 12, 1, 6 ),
	\, ( 3, 8, 1, 8 ),
	( 5, 6, 1, 6 ), \,	( 5, 8, 1, 8 ),\,
	( 17, 6, 1, 6).
\end{align*}
We then use Magma to search for the feasible subset $T$'s that satisfy the condition \eqref{eqn_II_psiDcond_HtoQM} in each case. The only transitive $m$-ovoids of the prescribed form are those listed in  Example \ref{exa_FR_HtoQM_smallpara}.

\subsection{The case $\cP=Q^+(d-1,q)$, $d\ge 6$}\label{subsec_II_QP}

In this case, we have $e=d-2$ by Table \ref{tab_valuede}. Since we assume that $e>2$, we have that $d\ge 6$. As in the case 2 of Lemma \ref{lem_extfld2468}, we can show that $H$ does not preserve an $\F_{q^{b'}}$-quadratic form of type $+$ with $b'>1$ by divisibility considerations. By Table \ref{tab_extfieldP'gt0}, we see that there are two possibilities:
\begin{enumerate}
	\item[(a)] $b$ even, $qb/2$ odd, $\cP'=Q(d/b-1,q^b)$, $\kappa'=Q'$, $\kappa=Q=\tr_{q^b/q}\circ \lambda Q'$ for some $\lambda\in\F_{q^2}$;
	\item[(b)] $b\equiv 2\pmod{4}$, $d/2$ even, $H'(x)=\kappa'(x,x)$, $\kappa=Q=\tr_{q^b/q}\circ H'$, and $\cP'=H(d/b-1,q^b)$.
\end{enumerate}
In both cases, the ovoid number $\theta_{d/2}=q^{d/2-1}+1$ of $Q^{+}(d-1,q)$ divides the order of $\Gamma(V',\kappa')$ only if $b=2$ upon direct check. We thus assume that $b=2$ in the sequel.\\

Suppose that we are in the case (a) with $b=2$. By Lemma \ref{lem_K_primdiv} and \ref{lem_Kirr}, $K$ has order divisible by $\Phi_{ef}^*(p)$ and $K$ is irreducible on $V'$. By the choice of $b$, $K$ is not a subgroup of $\textup{GL}(V')$ of the extension field type. Also, $V'$ is a vector space over $\F_{q^2}$, and $d/b\ge 3$ by the fact $d\ge 6$. By \cite[Theorem 3.1]{bamberg2008overgroups}, there are three possible cases:
\begin{enumerate}
	\item[(a1)] $qd/2$ odd, $K^{(\infty)}=\Omega_{d/2}(q^2)$, and $\cP'=Q(d/2-1,q)$;
	\item[(a2)] $d=14$, $\cP'=Q(6,q^2)$, $K^{(\infty)}=\PSU_{3}(q)$ with $q=3^f$;
	\item[(a3)] $d=14$, $\cP'=Q(6,q^2)$, $K^{(\infty)}=G_{2}(q^2)$ with $q$ odd.
\end{enumerate}
In the first case (a1),  $\Omega_{d/2}(q^2)$ is transitive on $\cM_1=\{\la x\ra_{\F_q}:\,Q'(x)=0\}$. Take $\delta\in\F_{q^2}^*$ such that $\delta+\delta^q=0$. Then $Q(v)=0$ if and only if $\lambda Q'(v)\in\F_q\cdot\delta$. The group $\Omega_{d/2}(q^2)$ has two orbits on $\cP\setminus\cM_1$, one of which is $\{\la v\ra_{\F_q}\in\cP:\,\lambda Q'(v)\in\square\cdot\delta\}$. The map $g:\,v\mapsto \mu v$ with $\mu\in\F_{q^2}^*$ of order $2(q-1)$ lies in $\Gamma^\#(V',Q')$, and it swaps the two orbits.
The two orbits are thus of the same length, and none of them can be $m$-ovoids since otherwise $m=\frac{1}{2}q^{d/2-1}$ would not be an integer. On the other hand, $\GamO_{d/2}^\#(q^2)$ is transitive on both $\cM_1$ and $\cP\setminus \cM_1$, so they are $\frac{q^{d/2-1}-1}{q-1}$- and $q^{d/2-1}$-ovoids respectively.  \medskip

In the cases (a2) and (a3), $K^{(\infty)}=H^{(\infty)}$. Case (a2) leads to Example \ref{exa_FR_PSU3_QPlus7}, so we consider case (a3). By \cite[Theorem 7.1]{giudici2020subgroups}, we deduce that $G_2(q^2)$ is transitive on the vectors of $V$ that have the same nonzero $Q'$-value. It is transitive on $\cP'=\cQ(6,q^2)$ and has a maximal torus $T$ which acts transitively on the set $\{\lambda v:\,\lambda\in\F_{q^2}^*\}$ for some $\la v\ra_{\F_{q^2}}\in\cP'$ by \cite[p.123]{wilson2009finite}. We deduce that $G_2(q^2)$ is transitive on the nonzero vectors of $V$ that have $Q'$-value $0$. As in (a1), $\cM_1$ and $\cP\setminus \cM_1$ are the only proper $G_2(q^2)$-invariant transitive $m$-ovoids.\\

Suppose that we are in the case (b) with $b=2$. Since $d\geq6$ and $d/2$ is even, we have $d/2\ge 4$. As in the case (a), $K$ has order divisible by $\Phi_{(d/2-1)\cdot 2f}^*(p)$ and $K$ is irreducible on $V'$. By the choice of $b$, $K$ is not a subgroup of $\textup{GL}(V')$ of the extension field type. Also, $V'$ is a vector space over $\F_{q^2}$. By \cite[Theorem 4.2]{bamberg2008overgroups}, either $(p,(d/2-1)f)=(2,3)$,   or  it is one of the cases in Table \ref{tab_Thm42Spor}. We check by Magma \cite{Magma} that there is no transitive $m$-ovoid that is $K^{(\infty)}$-invariant in all cases in the table. The case $(p,(d/2-1)f)=(2,3)$, i.e., $(q,d)=(2,8)$, is handled in Example \ref{exa_Q+72}.
\begin{table}[!h]
	\begin{center}
		\caption{The sporadic irreducible cases in \cite[Theorem 4.2]{bamberg2008overgroups}}\label{tab_Thm42Spor}
		\begin{tabular}{ccc|ccc}\toprule
			$K^{(\infty)}$& $d/2$ & $q^2$ & $K^{(\infty)}$ & $d/2$ & $q^2$  \\ \hline
			$2^{\cdot}A_7$ & $4$ & $3^2$ & $2^\cdot\PSL_{2}(7)$ & $4$ & $3^2$\\
			$2^{\cdot}A_7$ & $4$ & $5^2$ & $2^\cdot\PSL_{2}(7)$ & $4$ & $5^2$\\
			$4^{\cdot}\PSL_{3}(4)$ & $4$ & $3^2$ & $3^\cdot M_{22}$ & $6$ & $2^2$ \\ \bottomrule
		\end{tabular}
	\end{center}
\end{table}


\subsection{The case $\cP=W(d-1,q)$}\label{subsec_II_W}

In this case, we have $e=d$ by Table \ref{tab_valuede}.  As in the case 8 of Lemma \ref{lem_extfld2468}, we can show that $K$ does not fix an $\F_{q^{b'}}$-Hermitian form, where $b'$ is an even divisor of $d$ such that $\frac{d}{b'}$ is even. By Table \ref{tab_extfieldP'gt0}, we see that there are three possibilities for $\cP'$:
\begin{enumerate}
	\item[(a)] $d/b$ even,  $\cP'=W(d/b-1,q^b)$,  $\kappa=\tr_{q^b/q}\circ \kappa'$;
	\item[(b)] $b$ even, $d/b$ odd, $\cP'=H(d/b-1,q^b)$, $\kappa=\tr_{q^d/q}\circ \lambda\kappa'$ with $\lambda+\lambda^{q^{b/2}}=0$.\\
\end{enumerate}

Suppose that we are in the case (a). By Lemma \ref{lem_K_primdiv}, $\Phi_{df}^*(p)$ divides $|K|$. If $d/b=2$, then $K\le\PSp_{2}(q^{d/2})=\PSL_2(q^{d/2})$. We examine the list of maximal subgroups of $\PSL_{2}(q^{d/2})$ in \cite{Kingsubgroup2005} and deduce that $\Phi_{df}^*(p)$ divides $|K|$ only if $K\leq \GamU_{1}(q^{d/2})$. This contradicts the choice of $b$, so we must have $d/b\ge 4$. By Lemma \ref{lem_Kirr}, $K$ is irreducible on $V'$. It is not a extension field subgroup by the choice of $b$. We apply \cite[Theorem 3.1]{bamberg2008overgroups} to $K\le\textup{GL}(V')$, and see that there are four cases:
\begin{enumerate}
	\item[(a1)] $q$ even, $K^{(\infty)}=\Omega^{\pm}_{d/b}(q^b)$.
	\item[(a2)] $K^{(\infty)}=\textup{Sp}_{d/b}(q^b)$. In this case, $\Sp_{d/b}(q^b)$ is transitive on the nonzero vectors of $V'$ by \cite[Proposition 3.2]{Grove2002class}, and we have no proper transitive $m$-ovoid.
	\item[(a3)] $d/b=4$, $q$ even, $K^{(\infty)}=\textup{Sz}(q^b)$. This leads to Example \ref{exa_FR_W_Sz}.
	\item[(a4)] $d/b=6$, $q$ even, $K^{(\infty)}=G_2(q^b)$. In this case, the group $G_2(q^b)$ is transitive on the points of $W(5,q^b)$ by \cite{cooperstein1981maximal}. By examining its maximal torus, we deduce that $G_2(q^b)$ is transitive on the nonzero vectors of $V'$ and thus we have no proper transitive $m$-ovoid.
\end{enumerate}
It remains to consider (a1). The ovoid number $q^{d/2}+1$ of $\cP$ does not divide $|\textup{P}\GamO^{+}_{d/b}(q^b)|$, so we must have  $K^{(\infty)}=H^{(\infty)}=\Omega^{-}_{d/b}(q^b)$.
By \cite[Lemma 2.10.5]{kleidman1990subgroup}, $\Omega^{-}_{d/b}(q^b)$ is transitive on the nonzero vectors of the same $Q'$-value.  The set $\cM_1=\{\la v\ra_{\F_q}:\,Q'(v)=0\}$ is a transitive $\frac{q^{d/2-b}-1}{q-1}$-ovoid of $W(d-1,q)$ by Examples \ref{exa_Classical} and \cite{KellyCons}. We consider the case where $\cM\subseteq\cP\setminus\cM_1$. We take the same model of $Q'$ as in Section \ref{subsec_QMdbeven},  and assume that $q$ is even. The polar form $\kappa'$ of $Q'$ is alternating, and $\kappa=\tr_{q^b/q}\circ \kappa'$. Since $\cM$ is the union of $\Omega^{-}_{d/b}(q^b)$-orbits, it must be of the form
\begin{equation}\label{eqn_II_Mdef}
	\cM=\{\la v\ra_{\F_q}:\,Q'(v)\in\mathcal{T}\}
\end{equation}
for an $\F_q^*$-invariant subset $\mathcal{T}$ of $\F_{q^b}^*$.  Let $T$ be a set of $\F_q^*$-coset representatives of $\mathcal{T}$, so that $\mathcal{T}=\F_q^*\cdot T$. We have $|\cM|=m(q^{d/2}+1)$ with $m=|T|q^{d/2-b}$ by Lemma \ref{lem_QMinus_Msize}. The induced action of $H$ on $\{\F_q^*\cdot t:\,t\in T\}$ is transitive, so $|T|$ divides $bf$. We set $D=\{x\in V:\,Q'(x)\in\mathcal{T}\}$ correspondingly. As in the proof of Proposition \ref{prop_Ellip_Tbound}, we deduce that
\[
\psi_a(D)=-{q^{d/2-b}}\sum_{t\in\mathcal{T}}K(\psi_{\F_{q^b}},t,Q'(a))
\]
for a nonzero vector $a$. By \eqref{eqn_SNEFCHHpsiDM}, $\cM$ is a $|T|q^{d/2-b}$-ovoid if and only if
\begin{equation}\label{eqn_SNEFCQWKloos}
	\sum_{t\in\mathcal{T}}K(\psi_{\F_{q^b}},t,u)
	=\left\{\begin{array}{ll}q^{b}-|\mathcal{T}|, & \text{if }u\in\mathcal{T},\\
		-|\mathcal{T}|, &  \text{if }u\in\F_{q^b}\setminus \mathcal{T}.\end{array}\right.
\end{equation}
It holds trivially when $u=0$. In the case $b=2$, we have the following theorem by exactly the same argument as in the proof of Theorem \ref{thm_II_HtoWcons}.
\begin{thm}\label{thm_II_Wqevenb2cons}
	Take notation as above, and suppose that $p=2$ and $b=2$. Let $\cM$ be as in \eqref{eqn_II_Mdef} for an $\F_q^*$-invariant subset $\mathcal{T}$ of $\F_{q^2}^*$. Then $\cM$ is a $\frac{|\mathcal{T}|}{q-1}q^{d/2-2}$-ovoid of $W(d-1,q)$ if and only if $\mathcal{T}=\{x^{-1}:\,x\in\mathcal{T}\}$.
\end{thm}

\begin{remark}\label{rem_II_HtoW1}
	There are indeed transitive $m$-ovoids arising from Theorem \ref{thm_II_Wqevenb2cons}. For instance, if we set $\mathcal{T}=\{\lambda \gamma^{2^i}:\,\,0\le i\le 2f-1,\lambda\in\F_q^*\}$ with any specific $\gamma\in\F_{q^2}^*$,  then the condition is satisfied and we obtain a transitive $\frac{|\mathcal{T}|}{q-1}q^{d/2-2}$-ovoid of $W(d-1,q)$. In particular, if $\gamma$ is primitive, then $|\mathcal{T}|=2f(q-1)$ and we obtain a $2fq^{d/2-2}$-ovoid.
\end{remark}

\begin{example}\label{exa_FR_QMtoW_smallpara}
Take notation as above, and suppose $p=2$. In the case $(b,f,|T|)$ is one of $\{(3,1,3),\,(3,3,9),(4,1,4)\}$, there is an element $\gamma$ of $\F_{q^{b}}$ such that $\mathcal{T}=\{\gamma^{2^i}z:\,0\le i\le bf-1,z\in\F_q^*\}$ has size $bf(q-1)$ and \eqref{eqn_SNEFCQWKloos} holds. In these three cases, $\cM=\{\la v\ra_{\F_q}:\,Q'(x)\in\mathcal{T}\}$ is a $bfq^{d/2-b}$-ovoid of $W(d-1,q)$ if $d/b$ is even and $d/b\ge 4$. They are projectively equivalent to the $m$-ovoids arising from those of $Q^-(d-1,q)$ in Example \ref{exa_FR_QMtoQM_smallpara}, cf. (2) of Remark \ref{rem_2Fact}.
\end{example}
In the case $b\ge 3$, \eqref{eqn_SNEFCQWKloos} gives the same bound on $|T|$ as in Proposition \ref{prop_Ellip_Tbound}, i.e., $|T|\ge\frac{ q^b}{(q-1)(2q^{b/2}+1)}$. Set $y:=bf$. Together with the fact that $|T|$ divides $bf$, we deduce that
\[
2^y=q^b\le bf(2p^{bf/2}+1)(q-1)\le y(2p^{y/2}+1)(p^{y/3}-1)<2y\cdot 2^{5y/6}.
\]
It follows that $2^{y/6}<2y$, and so $y\le 37$. There are $24$ feasible $(b,f)$ tuples and correspondingly $39$ feasible $(b,f,|T|)$ tuples with $b\ge 3$. Among them there are $15$ tuples with $b>3$, and we check with Magma that there are transitive $m$-ovoids only in the case $(p,b,f,|T|)=(2,4,1,4)$. For the $24$ remaining tuples with $b=3$, the only ones that satisfy the bound on $|T|$ in Theorem \ref{thm_II_new} are  $(p,b,f,|T|)=(2,3,1,3)$, $(2,3,3,9)$.  The resulting transitive $m$-ovoids are exactly those in Example \ref{exa_FR_QMtoW_smallpara} by computer search. \medskip

From now on, suppose that we are in the case (b). In this case, $d/b$ is odd. First consider the case $d=b$. The case $d=4$ leads to Theorem \ref{thm_PGamSpTransi} (b), cf. Remark \ref{rem_end_sec4} (2). Suppose that $d\ge 6$ in the sequel.  By the same argument as in the case $d=b$ in the case $(b)$ of Section \ref{subsec_II_QM}, we obtain the same list of $35$ feasible tuples. By the same search strategy as therein, we use Magma to exclude all the cases except for the tuples $(2,6,1,3),\, (2,6,3,9)$ and $(2,6,8,24),(2,6,9,27)$. For the first two tuples, the only transitive $m$-ovoids are $Q^-(d-1,q)$ which are covered by Example \ref{exa_Classical}. The remaining two tuples are beyond our search capacity.

We next consider the case where $d/b\ge 3$. By Lemma \ref{lem_Kirr} and Lemma \ref{lem_K_primdiv}, $K$ is irreducible on $V'$ and has order divisible by $\Phi_{df}^*(p)$, and it is not a extension field subgroup by the choice of $b$.  We apply \cite[Theorem 4.1]{bamberg2008overgroups} to $K\le\textup{GL}(V')$ and see that either $K^{(\infty)}=\textup{SU}_{d/b}(q^{b/2})$ or it is one of the cases in Table \ref{tab_Thm4.1Spor}.
We check with Magma that there are only two examples for the cases in Table \ref{tab_Thm4.1Spor}:
\begin{enumerate}
	\item[(1)] a transitive $6$-ovoid of $W(5,5)$ stabilized by $3^\cdot A_7$,
	\item[(2)] a transitive $4$-ovoid of $W(5,3)$ stabilized by $2^{\cdot}\PSL_2(7)$.
\end{enumerate}
They are cases (c1) and (c2) of Theorem \ref{thm_PGamSpTransi}. We then consider the case $K^{(\infty)}=\textup{SU}_{d/b}(q^{b/2})$. In this case, $H^{(\infty)}=\textup{SU}_{d/b}(q^{b/2})$.  The group $\textup{SU}_{d/b}(q^{b/2})$ is transitive on nonzero singular vectors of $V'$ that has the same $H'$-value. In particular, $\cM_1=\{\la v\ra_{\F_q}:\,H'(v)=0\}$ forms an orbit, and it is a $\frac{q^{d/2-b}-1}{q-1}$-ovoid of $W(d-1,q)$ by Example \ref{exa_Extension} and \cite{KellyCons}. We now consider the case $\cM\subseteq\cP\setminus\cM_1$. Then $\cM$ takes the same form as in Section \ref{subsec_HtoW_I}. Take the same notation therein, and set $\epsilon=1/2$ or $1$ according as $q$ is odd or even. By Lemma \ref{lem_HtoW_Tbound}, $|T|$ divides $bf/2$ or $bf$ according as $q$ is even or odd, and $\mathcal{T}$ is $\F_q^*$-invariant if $q$ is odd and $|T|$ does not divide $\frac{1}{2}bf$. If $\cM$ is an $m$-ovoid with $m=\epsilon|T|q^{(d-b)/2}$,  then $|T|\ge\frac{q^{b/2}}{\epsilon(q-1)(2q^{b/4}+1)}$ by Proposition \ref{subsec_HtoW_I}. This gives the same inequality as in \eqref{eqn_II_HtoQM_ine}. The bound is trivial for $b=2,4$. We split into three cases.
\begin{enumerate}
	\item[(i)] In the case $b=2$, we must have $\cM=\cP\setminus\cM_1$ by the discussion in Section \ref{subsec_HtoW_I}, and it is indeed a transitive $q^{(d-b)/2}$-ovoid. In the case $b=4$ and $q$ is even, we are in the situation of Theorem \ref{thm_II_HtoWcons} and there are transitive examples as in Remark \ref{rem_II_HtoW1}.
	\item[(ii)] Consider the case where $b=4$ and $q$ is odd. Since $H$ is a subgroup of $\Gamma^\#(V',\kappa')$, cf. \eqref{eqn_GamV'def}, we have $\lambda(g)\in\F_q^*$ for each $g\in H$. The induced action of $H$ on $\{\square\cdot t:\,t\in T\}$ is transitive and elements of $\F_q^*$ are all squares of $\F_{q^2}^*$, so the elements of $\mathcal{T}=\square\cdot T$ are all squares or nonsquares of $\F_{q^2}^*$. This case is excluded by Lemma \ref{lem_II_HtoW_qodd}.
	\item[(iii)]In the case $b\ge 6$, we obtain the same list of feasible $(p,b,f, |T|)$ tuples as in the last paragraph of Section \ref{subsec_II_Hdood}. When $p$ is odd, there are $6$ out of $31$ tuples such that $|T|$ divides $\frac{bf}{2}$:
$	( 3, 6, 3, 9 )$, $( 3, 6, 2, 6 )$, $( 3, 6, 4, 12 )$, $( 5, 6, 1, 3 )$,
$( 3, 6, 1, 3 )$, $( 5, 6, 2, 6 )$. We check with Magma that there are no transitive $m$-ovoids in such cases. For the remaining $25$ tuples with $q$ odd, $|T|$ does not divide $\frac{bf}{2}$ and so is  $\F_q^*$-invariant by Lemma \ref{lem_HtoW_Tbound}. After applying the bound on $|T|$ in Theorem \ref{thm_II_new}, there are $5$ tuples that remain: $(3,6,1,6), (3,8,1,8), (5,8,1,8), (3,8,2,16), (3,10,1,10)$. Similarly, there are $18$ out of $39$ tuples that satisfy the bound on $|T|$ in Theorem \ref{thm_II_new} when $q$ is even. We use Magma to search for possible $T$'s that satisfy the necessary and sufficient condition \ref{eqn_II_psiDcond_HtoW} for each surviving tuple. The resulting transitive $m$-ovoids are those in Example \ref{exa_FR_HtoW_smallpara}.
\end{enumerate}

This concludes our discussion in the extension field case. The main results in Theorems \ref{thm_PGamUTransi}-\ref{thm_PGamSpTransi} now follow by summarizing all the cases in Section \ref{sec:mainresults} and Section  \ref{sec_extField}.


\section{Concluding Remarks}\label{sec:Conclud}
In this paper, we classify the $m$-ovoids that admit a transitive automorphism group that acts irreducibly on the underlying vector space. We follow the line of \cite[Theorem 3.1]{bamberg2008overgroups} which classifies the subgroups of finite classical groups whose orders are divisible by certain primitive divisors based on \cite{Guralnick1999Linear}. The main results are presented in Theorems \ref{thm_PGamUTransi}-\ref{thm_PGamSpTransi}, and the transitive $m$-ovoids are summarized in the Tables \ref{tab_sumHerm}-\ref{tab_sumW} and Remark \ref{rem_end_sec4}. We obtain many interesting sporadic examples and unexpected new infinite families with large stabilizers. In particular, we develop a theory that utilizes exponential sum calculations to study $m$-ovoids of a specific form in Section \ref{sec:boundsection}, which generalizes the field reduction approach \cite{KellyCons}.

We conclude this paper with some open problems:
\begin{enumerate}
\item[(1)] It will be of particular interest to further study the $m$-ovoids in Section \ref{sec:boundsection}. We are only concerned with the transitive ones in this paper, but it is very likely that there will be more examples if we drop the transitivity assumption. Also, it is desirable to develop a similar theory for $i$-tight sets and obtain new infinite families.
\item[(2)] As stated in Remark \ref{rem_end_sec4} (2), we did not attempt to classify the transitive $m$-ovoids of $W(3,q)$ whose stabilizers lie in $\GamU^\#_1(q^2)$. There should be more such $m$-ovoids, as is indicated by the examples in \cite{cossidente2008m}. We leave this as an open problem, which we expect to be complicated and difficult. We did not consider the transitive $m$-ovoids of $Q^+(3,q)$ either for lack of geometric significance. By (3) of the same remark, there are two parameter sets in $W(5,q)$ that are left undecided.
\item[(3)] It is a natural next step to classify the transitive $m$-ovoids whose automorphism group is reducible on the ambient vector space. However, it will be considerably more complicated than the irreducible case, as is indicated by the classification of transitive ovoids in the Hermitian case \cite{FengLi}.
\end{enumerate}


\section*{Acknowledgements}{This work was supported by National Natural Science Foundation of China (Grant No. 12171428), the Sino-German Mobility Programme M-0157 and Shandong Provincial Natural Science Foundation (Grant No. ZR2022QA069). The authors are grateful to the reviewers for their detailed comments and suggestions that helped to improve the presentation of the paper greatly.}




\end{document}